\documentclass[10pt]{article}%
\usepackage{amssymb}
\usepackage[usenames]{color}
\usepackage{amsmath}
\usepackage{amsfonts}
\usepackage{amssymb}
\usepackage{graphicx}%
\providecommand{\U}[1]{\protect\rule{.1in}{.1in}}
\newtheorem{theorem}{Theorem}

\newtheorem{definition}[theorem]{Definition}

\newtheorem{lemma}[theorem]{Lemma}

\newtheorem{proposition}[theorem]{Proposition}

\newenvironment{proof}[1][Proof]{\noindent\textbf{#1.} }{\ \rule{0.5em}{0.5em}}
\begin{document}

\title{Symmetries and choreographies in families bifurcating from \\the polygonal relative equilibrium of the $n$-body problem}
\author{Renato Calleja\thanks{Instituto de Investigaciones en Matem\'{a}ticas Aplicadas y en Sistemas, Universidad Nacional Aut\'{o}noma de M\'{e}xico, calleja@mym.iimas.unam.mx}, Eusebius Doedel\thanks{Department of Computer Science, Concordia University, Montreal, Canada, doedel@cs.concordia.ca}, Carlos
Garc\'{\i}a-Azpeitia\thanks{Facultad de Ciencias, Universidad Nacional Aut\'{o}noma de M\'{e}xico, cgazpe@ciencias.unam.mx}}
\maketitle

\begin{abstract}
We use numerical continuation and bifurcation techniques in a boundary value
setting to follow Lyapunov families of periodic orbits. These arise from the
polygonal system of $n$ bodies in a rotating frame of reference. When the
frequency of a Lyapunov orbit and the frequency of the rotating frame have a
rational relationship then the orbit is also periodic in the inertial frame.
We prove that a dense set of Lyapunov orbits, with frequencies satisfying a
diophantine equation, correspond to choreographies. We present a sample of the
many choreographies that we have determined numerically along the Lyapunov
families and along bifurcating families, namely for the cases $n=4,~6,~7,~8$,
and $9$. We also present numerical results for the case where there is a
central body that affects the choreography, but that does not participate in
it. Animations of the families and the choreographies can be seen at the link
below\footnote{\texttt{http://mym.iimas.unam.mx/renato/choreographies/index.html}%
}.

\end{abstract}


\section*{Introduction}

The study of $n$ equal masses that follow the same path has attracted much
attention in recent years. The first solution that differs from the classical
Lagrange circular one was discovered numerically by C.~Moore in 1993
\cite{Mo93}, where three bodies follow one another around the now famous
figure-eight orbit. This orbit was located by minimizing the action among
symmetric paths. Independently in \cite{ChMo00}, Chenciner and Montgomery
(2000) gave a rigorous mathematical proof of the existence of this orbit, by
minimizing the action over paths that connect a colinear and an isosceles
configuration. Such solutions are now commonly known as \textquotedblleft
choreographies\textquotedblright, after the work in \cite{Si00}, where
C.~Sim\'{o} presented extensive numerical computations of choreographies for
many choices of the number of bodies.

The results in \cite{ChMo00} mark the beginning of the development of
variational methods, where the existence of choreographies can be associated
with the problem of finding critical points of the classical action of the
Newton equations of motion. The main obstacles encountered in the application
of the principle of least action are the existence of paths with collisions,
and the lack of compactness of the action. In \cite{FeTe04}, Terracini and
Ferrario (2004) applied the principle of least action systematically over
symmetric paths to avoid collisions, using ideas introduced by Marchal
\cite{Ma02}. For the discussion of these and other variational approaches we
refer to \cite{BaTe04,Ch03,Fe06,FePo08,TeVe07}, and references therein.

Another way to obtain choreographies is by using continuation methods.
Chenciner and F\'{e}joz (2009) pointed out in \cite{ChFe08} that
choreographies appear in dense sets along the Vertical Lyapunov families that
arise from $n$ bodies rotating in a polygon; see also
\cite{ChFe05,GaIz13,Ma00}. The local existence of the Vertical Lyapunov
families is proven in \cite{ChFe08} using the Weinstein-Moser theory. When the
frequency varies along the Vertical Lyapunov families then an infinite number
of choreographies exists; a fact established in \cite{ChFe08} for orbits close
to the polygon equilibrium, with $n\leq6$. While similar computations can be
carried out for other values of $n$, a general analytical proof that is valid
for all $n$ remains an open problem.

In \cite{GaIz13} C.~Garc\'{\i}a-Azpeitia and J. Ize (2013) proved the global
existence of bifurcating Planar and Vertical Lyapunov families, using the
equivariant degree theory from \cite{IzVi03}. The purpose of our current work
is to compute such global families numerically, as well as subsequently
bifurcating families. To explain our numerical results in a precise notational
setting we first recall some relevant results from \cite{GaIz13}.

The equations of motion of $n$ bodies of unit mass in a rotating frame are
given by
\begin{align}
\ddot{u}_{j}+2\sqrt{s_{1}}~i~\dot{u}_{j}  &  =s_{1}u_{j}-\sum_{i=1(i\neq
j)}^{n}\frac{u_{j}-u_{i}}{\left\Vert (u_{j},z_{j})-(u_{i},z_{i})\right\Vert
^{3}}~,\label{NE}\\
\ddot{z}_{j}  &  =-\sum_{i=1(i\neq j)}^{n}\frac{z_{j}-z_{i}}{\left\Vert
(u_{j},z_{j})-(u_{i},z_{i})\right\Vert ^{3}}~,\nonumber
\end{align}
where the $(u_{j},z_{j})\in\mathbb{C}\times\mathbb{R}$ are the positions of
the bodies in space, and $s_{1}$ is defined by%
\begin{equation}
s_{k}=\frac{1}{4}\sum_{j=1}^{n-1}\frac{\sin^{2}(kj\zeta/2)}{\sin^{3}%
(j\zeta/2)}~,\qquad\zeta=\frac{2\pi}{n}~. \label{S}%
\end{equation}
The circular, polygonal relative equilibrium consists of the positions
\begin{equation}
u_{j}=e^{ij\zeta},\qquad z_{j}=0\text{.} \label{RE}%
\end{equation}
The frequency of the rotational frame is chosen to be $\sqrt{s_{1}}$, so that
the polygon (\ref{RE}) is an equilibrium of (\ref{NE}). The emanating Lyapunov
families have starting frequencies that are equal to the natural modes of
oscillation of the equilibrium (\ref{RE}). These Lyapunov families constitute
continuous families in the space of renormalized $2\pi$-periodic functions.
The \emph{global} property means that the norm or the period of the orbits
along the family tends to infinity, or that the family ends in a collision or
at a bifurcation orbit.

The theorem in \cite{GaIz13} states that for $n\geq6$ and for each integer $k$
such that
\[
3\leq k\leq n-3,
\]
the polygonal relative equilibrium has one \emph{global bifurcation of planar
periodic solutions} with symmetries%
\begin{equation}
u_{j}(t)=e^{ij\zeta}u_{n}(t+jk\zeta),\qquad u_{n}(t)=\bar{u}_{n}(-t)\text{.}
\label{PS}%
\end{equation}
Moreover, the proof in \cite{GaIz13} predicts solutions with $k=2$ or $n-2$ if
the linear equations at the polygonal equilibrium have normal modes
corresponding to these symmetries. In fact, three cases occur for different
values of $n$: for $n=4,5,6$ there are no solutions with $k=2$ or $n-2$, for
$n=7,8,9$ there are two solutions with $k=2$ and no solutions with $k=$ $n-2$,
and for $n\geq10$ there is one solution with $k=2$ and one with $k=$ $n-2$.

In the case of spatial Lyapunov families the eigenvalues of the linearized
system of equations are given explicitly by $i\sqrt{s_{k}}$, for
$k=1,...,n-1$; see \cite{ChFe08} and \cite{GaIz13}. The eigenvalues
$i\sqrt{s_{k}}$ are resonant due to the fact that $s_{n-k}=s_{k}$ for $1\leq
k<n/2$. Moreover, the first eigenvalue $i\sqrt{s_{1}}$ is resonant with the
triple planar eigenvalue $i\sqrt{s_{1}}$, and hence is highly degenerate.
These resonances can be dealt with using the equivariant degree theory in
\cite{IzVi03}.

The theorem in \cite{GaIz13} states that for $n\geq3$ and for each $k$ such
that
\[
1\leq k\leq n/2\text{,}%
\]
the polygonal relative equilibrium has one\emph{ global bifurcation of spatial
periodic solutions}, which start with frequency $\sqrt{s_{k}}$, have the
symmetry (\ref{PS}), as well as the symmetries
\begin{equation}
z_{j}(t)=z_{n}(t+jk\zeta), \label{SS}%
\end{equation}
and
\begin{equation}
u_{n}(t)=u_{n}(t+\pi),\qquad z_{n}(t)=-z_{n}(t+\pi). \label{8}%
\end{equation}
For example, for the case where $k=n/2$ and $n$ is even, we have $k\zeta=\pi$.
Then the symmetries (\ref{PS}), (\ref{SS}) and (\ref{8}) imply that%
\begin{align*}
u_{j}(t)  &  =e^{ij\zeta}u_{n}(t+j\pi)=e^{ij\zeta}u_{n}(t),\\
z_{j}(t)  &  =z_{n}(t+jk\zeta)=(-1)^{j}z_{n}(t).
\end{align*}
Solutions having these symmetries are known as Hip-Hop orbits, and have been
studied in \cite{BaCo06,DaTr83,MeSc93,TeVe07}.

Solutions with symmetries (\ref{PS}) and (\ref{SS}) are \textquotedblleft
traveling waves\textquotedblright\, in the sense that each body follows the
same path, but with a rotation and a time shift. The symmetries allow us to
establish that a dense set of solutions along the family are choreographies in
the inertial frame of reference.

We say that a planar or spatial Lyapunov orbit is $\ell:m$\emph{\ resonant} if
its period and frequency are
\[
T=\frac{2\pi}{\sqrt{s_{1}}}\left(  \frac{\ell}{m}\right)  \text{,\qquad}%
\nu=\sqrt{s_{1}}\frac{m}{\ell}\text{,}%
\]
where $\ell$ and $m$ are relatively prime, and such that%
\[
k\ell-m\in n\mathbb{Z}.
\]
In Theorem~\ref{proposition} we prove that $\ell:m$ resonant Lyapunov orbits
are choreographies in the inertial frame. Each of the integers $k$, $\ell$,
and $m$ plays a different role in the description of the choreographies.
Indeed, the projection of the choreography onto the $xy$-plane has winding
number $\ell$ around a center, and is symmetric with respect to the
$\mathbb{Z}_{m}$-group of rotations by $2\pi/m$. In addition the $n$ bodies
form groups of $d$-polygons, where $d$ is the greatest common divisor of $k$
and $n$.

Some choreographies wind around a toroidal manifold with winding numbers
$\ell$ and $m$, \textit{i.e.}, the choreography path is a $(\ell,m)$-torus
knot. In particular, such orbits appear in families that we refer to as "Axial
families", \textit{e.g.}, in Figure~\ref{fig08}. In \cite{BaTe04} and
\cite{Mo} different classifications for the symmetries of planar
choreographies have been presented. These classifications differ from the one
presented here since they are designed for choreographies found by means of a
variational approach. The nature of our approach is continuation and as such,
the winding numbers $\ell$ and $m$ appear in a natural manner in the
classification of the choreographies. Therefore, our approach presents
complementary information not available with variational methods. We note that
for other values of $\ell$ and $m$ the orbits of the $n$ bodies in the
inertial frame are also closed, but consist of multiple curves, called
``\textit{multiple choreographic solutions}'' in \cite{Ch03}.

We use robust and highly accurate boundary value techniques with adaptive
meshes to continue the Lyapunov families. An extensive collection of python
scripts that reproduce the results reported in this article for a selection of
values of $n$ will be made freely available. These scripts control the
software AUTO to carry out the necessary sequences of computations. Similar
scripts will be available for related problems, including an $n$ vortex
problem and a periodic lattice of Schr\"odinger sites.

In \cite{ChFe08} the numerical continuation of the Vertical Lyapunov families
is implemented as local minimizers in subspaces of symmetric paths. Presumably
not all families are local minimizers restricted to subspaces. One advantage
of our procedure is that it allows the numerical continuation of all Planar
and Vertical Lyapunov families that arise from simple eigenvalues. The
systematic computation of periodic orbits that arise from eigenvalues of
higher multiplicity remains under investigation. Previous numerical work has
established the existence of many choreographies; see for example
\cite{ChGe02}. Computer-assisted proofs of the existence of choreographies
have been given in, for example, \cite{KaZg03} and \cite{KaSi07}. It would be
of interest to use such techniques to mathematically validate the existence of
some of the choreographies in our article. The figure-8 orbit is still the
only choreography known to be stable \cite{KaSi07}, and so far we have not
found evidence of other stable choreographies.

In Section~1 we prove that a dense set of orbits along the Lyapunov families
corresponds to choreographies. In Section 2 we describe the numerical
continuation procedure used to determine the periodic solution families, and
in Section 3 we give examples of numerically computed Lyapunov families and
some of their bifurcating families. In Section 4 we provide a sample of the
choreographies that appear along Planar Lyapunov families. Section 5 presents
choreographies along the Vertical Lyapunov families and along their
bifurcating families. In particular, a family of axially symmetric orbits
forms a connection between a Vertical family and a Planar family.
Choreographies along such tertiary Planar families are referred to as
``unchained polygons'' in \cite{ChFe08}.

In Section 6 we present results for a similar configuration, namely the
Maxwell relative equilibrium, where a central body is added at the center of
the $n$-polygon. This configuration has been used as a model to study the
stability of the rings of Saturn, as established in \cite{Mo92} and in
\cite{GaIz13,VaKo07,Ro00} for $n\geq7$. Using a similar approach as in the
earlier sections, we determine solutions where $n$ bodies of equal mass $1$
follow a single trajectory, but with an additional body of mass $\mu$ at or
near the center. While this extra body does not participate in the
choreography, it does affect its structure, and its stability properties. We
also present Vertical Lyapunov families that bifurcate from a non-circular,
polygonal equilibrium, whose solutions have symmetries that correspond to
\textquotedblleft standing waves\textquotedblright, and which do not give rise
to choreographies in the inertial frame.

\section{Choreographies and Lyapunov Families}

In this section we prove that there are Lyapunov orbits of the $n$ body
problem that correspond to choreographies in the inertial frame of reference.

\begin{lemma}
\label{Lemma}Let
\[
\Omega=\frac{1}{n}\left(  k\frac{\sqrt{s_{1}}}{\nu}-1\right)  \text{.}%
\]
Then in the inertial frame of reference, with period scaled to $2\pi$, the
Planar Lyapunov orbits satisfy
\[
q_{j}(t)=e^{-ij(2\pi)\Omega}q_{n}(t+jk\zeta).
\]

\end{lemma}

\begin{proof}
: In the inertial frame the solutions are given by%
\[
q_{j}(t)=e^{i\sqrt{s_{1}}t}u_{j}(\nu t),
\]
where $\nu$ is the frequency and $T=2\pi/\nu$ is the period. Reparametrizing
time the solution becomes $q_{j}(t)=e^{it\sqrt{s_{1}}/\nu}u_{j}(t)$, where
$u_{j} $ is the $2\pi$-periodic solution with the symmetries (\ref{PS}). We
have%
\[
q_{j}(t)=e^{it\sqrt{s_{1}}/\nu}u_{j}(t)=e^{it\sqrt{s_{1}}/\nu}e^{ij\zeta}%
u_{n}(t+jk\zeta)\text{.}%
\]
Since
\[
q_{n}(t+jk\zeta)=e^{i(t+jk\zeta)\sqrt{s_{1}}/\nu}u_{n}(t+jk\zeta)\text{,}%
\]
it follows that
\[
q_{j}(t)=e^{it\sqrt{s_{1}}/\nu}e^{ij\zeta}\left(  e^{-i(t+jk\zeta)\sqrt{s_{1}%
}/\nu}q_{n}(t+jk\zeta)\right)  =e^{-ij\zeta n\Omega}q_{n}(t+jk\zeta).
\]

\end{proof}

In particular, if $\Omega\in\mathbb{Z}$ then the Lyapunov solutions satisfy
\begin{equation}
q_{j}(t)=q_{n}(t+jk\zeta)\text{,} \label{qn}%
\end{equation}
and are choreographies. In fact, planar choreographies exists for any rational
number $\Omega=p/q$ where $q$ is relatively prime to $n$.

\begin{proposition}
\label{PC}If $\Omega=p/q$, with $q$ relatively prime to $n$, then
\begin{equation}
q_{j}(t)=q_{n}(t+j\left(  1_{n}k\zeta\right)  )\text{,}%
\end{equation}
where $1_{n}=1$ mod $n$. The solution $q_{n}(t)$ is $2\pi m$-periodic, where
$m$ and $\ell$ are relatively prime such that%
\begin{equation}
\frac{\ell}{m}=\frac{np+q}{kq}\text{.} \label{lm}%
\end{equation}

\end{proposition}

\begin{proof}
: If $\Omega=p/q$, the solution satisfies
\begin{equation}
q_{j}(t)=e^{-i2\pi jp/q}q_{n}(t+jk\zeta).
\end{equation}
Since $n$ and $q$ are relatively prime, we can define $q^{\ast}$ as the
modular inverse of $q$. Setting $1_{n}=q^{\ast}q$, there is an $\ell$ such
that $j 1_{n}=j+n\ell$ for any $j$. Then we have%
\begin{equation}
q_{j}(t)=q_{j+n\ell}(t)=e^{-i2\pi(j 1_{n}p/q}q_{n}(t+j(1_{n}k\zeta
))=e^{-i2\pi(jq^{\ast}p)}q_{n}(t+j(1_{n}k\zeta))\text{.}%
\end{equation}
Since
\[
\frac{\sqrt{s_{1}}}{\nu}=\frac{n\Omega+1}{k}=\frac{np+q}{qk}=\frac{\ell}%
{m}\text{,}%
\]
it follows that $e^{it\sqrt{s_{1}}/\nu}$ is $2\pi m$-periodic, and since
$u_{n}(t)$ is $2\pi$-periodic, we also have that the function $q_{n}%
(t)=e^{it\sqrt{s_{1}}/\nu}u_{n}(t)$ is $2\pi m$-periodic. 

\end{proof}

\begin{proposition}
\label{SC}For $\Omega=p/q$, with $q$ and $n$ relatively prime, the spatial
Lyapunov solution is a choreography that satisfies
\begin{equation}
(q_{j},z_{j})(t)=(q_{n},z_{n})(t+j(1_{n}k\zeta))\text{,}%
\end{equation}
where $1_{n}=1$ mod $n$ and $(q_{j},z_{j})(t)$ is $2\pi m$-periodic.
\end{proposition}

\begin{proof}
: For the planar component of the spatial Lyapunov families we have
$q_{j}(t)=q_{n}(t+j1_{n}k\zeta)$, where $q_{n}(t)$ is $2\pi m$-periodic. We
have in addition that the spatial component $z_{n}$ is $2\pi$-periodic and
satisfies $z_{j}(t)=z_{n}(t+jk\zeta)$. Since $1_{n}=1$ mod $n$, we have
\[
z_{j}(t)=z_{n}(t+jk\zeta)=z_{n}(t+j1_{n}k\zeta),
\]
where $z_{n}(t)$ is also $2\pi m$-periodic. 

\end{proof}

For fixed $n$ the set of rational numbers $p/q$ such that $q$ and $n$ are
relatively prime is dense. If the range of the frequency $\nu$ along the
Lyapunov family contains an interval, then there is a dense set of rational
numbers $\Omega=p/q$ inside that interval. Hence there is an infinite number
of Lyapunov orbits that correspond to choreographies. To be precise, the
resonant Lyapunov orbit gives a choreography that has period
\[
mT=m\frac{2\pi}{\nu}=\frac{2\pi}{\sqrt{s_{1}}}\ell~\text{,}%
\]
where $T$ is the period of the resonant Lyapunov orbit. Furthermore, the
number $\ell$ is related to the number of times that the orbit of the
choreography winds around a central point. Rational numbers $p/q$, where $q$
is relatively prime to $n$, appear infinitely often in an interval, with $p$
and $q$ arbitrarily large. In such a frequency interval the infinite number of
rationals $p/q$ that correspond to choreographies give arbitrarily large
$\ell$ and $m$ as well. This gives rise to an infinite number of
choreographies, with arbitrarily large frequencies $\frac{2\pi}{\sqrt{s_{1}}%
}\ell$, and orbits of correspondingly increasing complexity.

Although the previous results give sufficient conditions for the existence of
infinitely many choreographies, there can be additional choreographies due to
the fact that the orbit of the choreography $q_{n}(t)$ has additional
symmetries by rotations of $2\pi/m$. We now describe these symmetries and the
necessary conditions.

\begin{definition}
We define a Lyapunov orbit as being $\ell:m$ resonant if it has period
\[
T_{\ell:m}=\frac{2\pi}{\sqrt{s_{1}}}\frac{\ell}{m}\text{,}%
\]
where $\ell$ and $m$ are relatively prime such that%
\[
k\ell-m\in n\mathbb{Z}\text{.}%
\]

\end{definition}

\begin{theorem}
\label{proposition}In the inertial frame an $\ell:m$ resonant Lyapunov orbit
is a choreography,%
\[
(q_{j},z_{j})(t)=(q_{n},z_{n})(t+j\tilde{k}\zeta)\text{,}%
\]
where $\tilde{k}=k-(k\ell-m)\ell^{\ast}$ with $\ell^{\ast}$ the $m$-modular
inverse of $\ell$. The projection on the $xy$-plane of the choreography is
symmetric by rotations of the angle $2\pi/m$ and winds around a center $\ell$
times. The period of the choreography is $m~T_{\ell:m}$.
\end{theorem}

\begin{proof}
: Since $u_{n}(t)$ is $2\pi$-periodic and
\[
e^{it\sqrt{s_{1}}/\nu}=e^{it\ell/m}%
\]
is $2\pi m$-periodic, the function $q_{n}(t)=e^{it\sqrt{s_{1}}/\nu}u_{n}(t)$
is $2\pi m$-periodic. Furthermore, since%
\begin{equation}
q_{n}(t-2\pi)=e^{-i2\pi\ell/m}q_{n}(t), \label{symq}%
\end{equation}
the orbit of $q_{n}(t)$ is invariant under rotations of $2\pi/m$. By
Lemma\ \ref{Lemma}, since
\[
\Omega=\frac{k\ell-m}{nm}=\frac{r}{m},
\]
with $r=(k\ell-m)/n\in\mathbb{Z}$, the solutions satisfy
\begin{equation}
q_{j}(t)=e^{-i2\pi j(r/m)}q_{n}(t+jk\zeta)\text{.}%
\end{equation}
Since $\ell$ and $m$ are relatively prime we can find $\ell^{\ast}$, the
$m$-modular inverse of $\ell$. Since $\ell\ell^{\ast}=1$ mod $m$, it follows
from the symmetry (\ref{symq}) that
\[
q_{n}(t-2\pi jr\ell^{\ast})=e^{-i2\pi j(r/m)}q_{n}(t).
\]
Therefore,%
\begin{equation}
q_{j}(t)=e^{-i2\pi j(r/m)}q_{n}(t+jk\zeta)=q_{n}(t+j(k-rn\ell^{\ast})\zeta).
\end{equation}

For the planar component $q_{j}(t)$ of spatial Lyapunov families we have the
same relation. In addition we have that the spatial component $z_{n}$ is
$2\pi$-periodic and satisfies $z_{j}(t)=z_{n}(t+jk\zeta)$. Since $rn\ell
^{\ast}\zeta=2\pi r\ell^{\ast}\in2\pi\mathbb{Z}$, it follows that
\[
z_{j}(t)=z_{n}(t+jk\zeta)=z_{n}(t+j(k-rn\ell^{\ast})\zeta),
\]
and thus $z_{n}(t)$ is also $2\pi m$-periodic. 

\end{proof}


\section{Numerical continuation of Lyapunov families}

To continue the Lyapunov families numerically it is necessary to take the
symmetries into account. The equations (\ref{NE}) in the rotational frame,
have two symmetries that are inherited from Newton's equations in the inertial
frame, namely rotations in the plane $e^{\theta i}u_{j}$ and translations in
the spatial coordinate $z_{j}+c$. This implies that any rotation in the plane
and any translation of an equilibrium is also an equilibrium, and that the
linear equations have two conserved quantities and two trivial eigenvalues.

To determine the conserved quantities, we can sum the equation (\ref{NE}) over
the $z_{j}$ coordinates to obtain that $\sum_{j=1}^{n}\ddot{z}_{j}=0$,
\textit{i.e.}, the linear momentum in $z$ is conserved%
\begin{equation}
\sum_{j=1}^{n}\dot{z}_{j}(t)= \text{constant.}%
\end{equation}
The other conserved quantity can be obtained easily in real coordinates.
Identifying $i$ with the symplectic matrix $J$, taking the real product of the
$u$ component of equation (\ref{NE}) with the generator of the rotations
$Ju_{j}$, and summing over $j$, we obtain%
\[
0=\sum_{j=1}^{n}\left\langle \ddot{u}_{j}+2\sqrt{s_{1}}J\dot{u}_{j}%
,Ju_{j}\right\rangle _{\mathbb{R}^{2}}=\frac{d}{dt}\sum_{j=1}^{n}\left\langle
\dot{u}_{j}+\sqrt{s_{1}}Ju_{j},Ju_{j}\right\rangle _{\mathbb{R}^{2}}\text{.}%
\]
Therefore, the second conserved quantity is%
\begin{equation}
\sum_{j=1}^{n}\dot{u}_{j}\cdot Ju_{j} - \sqrt{s_{1}}\left\vert u_{j}%
\right\vert ^{2} .\nonumber
\end{equation}

To continue the Lyapunov families numerically we need to take the conserved
quantities into account. Let $x_{j}=(u_{j,}z_{j})$ be the vector of positions
and $v_{j}=(\dot{u}_{j,}\dot{z}_{j})$ the vector of velocities. In our
numerical computations we use the augmented equations
\begin{align}
\label{EqA}\dot{x}_{j}  & = v_{j}\text{,}\nonumber\\
\\
\dot{v}_{j}  & = 2\sqrt{s_{1}}~diag(J,0)~v_{j}+\nabla_{x_{j}}V+ \sum_{k=1}%
^{3}\lambda_{k}F_{j}^{k} ,\nonumber
\end{align}
where $V(x)=\sum_{i<j}\left\Vert x_{j}-x_{i}\right\Vert ^{-1}$, and where
$F_{j}^{1}=e_{3}$ corresponds to the generator of the translations in $z$,
$F_{j}^{2}=diag(J,0)x_{j}$ to rotations in the plane, and $F_{j}^{3}=v_{j}$ to
the conservation of the energy. The solutions of the equation (\ref{EqA}) are
solutions of the original equations of motion when the values of the three
parameters $\lambda_{k}$ are zero. It is known that the converse of this
statement is also true, for instance see \cite{IzVi03} and \cite{DoVa03}.

\begin{proposition}
Assume that the functions $F^{k}=(F_{1}^{k},...,F_{n}^{k})$ for $k=1,2,3$, are
orthogonal (or linearly independent). Then a solution $(x,v)$ of the equation
is a solution of the augmented equation (\ref{EqA}) if and only if
$\lambda_{j}=0$ for $j=1,2,3$.
\end{proposition}

\begin{proof}
: Multiplying the equation in (\ref{EqA}) by $F_{j}^{k}$, summing over $j$,
and integrating by parts, we obtain
\[
\int_{0}^{2\pi}\sum_{j=1}^{n}\dot{v}_{j}\cdot F_{j}^{k}dt=\lambda_{k}\int
_{0}^{2\pi}\sum_{j=1}^{n}\left\vert F_{j}^{k}\right\vert ^{2}dt\text{.}%
\]
Suppose that $(x,v)$ is a solution. Then it conserves the aforementioned
quantities, and therefore
\[
\int_{0}^{2\pi}\sum_{j=1}^{n}\dot{v}_{j}\cdot F_{j}^{k}dt=0.
\]
The result that $\lambda_{j}=0$ then follows from the orthogonality of the
fields $F^{k}$. 

\end{proof}

For the purpose of numerical continuation the period of the solutions is
rescaled to $1$, so that it appears explicitly in the equations. Let
$\varphi(t,x,v)$ be the flow of the rescaled equations. Then we define the
time-$1$ map for the rescaled flow as
\[
\varphi(1,x,v;T,\lambda_{1},\lambda_{2},\lambda_{3}):\mathbb{R}^{6n}%
\times\mathbb{R}^{4}\rightarrow\mathbb{R}^{6n}.
\]
Let $\tilde{x}(t)$ be the solution computed in the previous step along a
family. We implement Poincar\'{e} restrictions given by the integrals%
\begin{align*}
I_{1}(x,v)  &  =\int_{0}^{1}x_{n}\cdot e_{2}~dt=0,\\
I_{2}(x,v)  &  =\int_{0}^{1}x_{n}\cdot e_{3}~dt=0,\\
I_{3}(x,v)  &  =\int_{0}^{1}\left(  x_{n}(t)-\tilde{x}_{n}(t)\right)
\cdot\tilde{x}_{n}^{\prime}(t)~dt=0\text{,}%
\end{align*}
which correspond to rotations, translations in $z$, and the energy, respectively.

The results in \cite{DoVa03} are based on the continuation of zeros of the
map
\[
F(x,v;T,\lambda_{1},\lambda_{2},\lambda_{3}):=\left(  (x,v)-\varphi
(x,v),I_{1},I_{2},I_{3}\right)  :\mathbb{R}^{6n+4}\rightarrow\mathbb{R}%
^{6n+3}.
\]
Actually, continuation is done with AUTO for the complete operator equation in
function space. That is, the numerical computation of the maps $\varphi$ and
$I_{j}$ is done for the corresponding operators in $C_{2\pi}^{2}%
(\mathbb{R}^{6n})$. This operator equation is discretized using highly
accurate piecewise polynomial collocation at Gauss points.

\section{Lyapunov families and bifurcating families}

In this section we give a brief description of some of the many solutions
families that we have computed using python scripts that drive the AUTO
software. We start with Planar families that arise from the circular,
polygonal equilibrium state of the $n$-body problem when $n\geq6$. For the
case $n=6$ there is a single such Planar family. While of interest, its orbits
are of relatively small amplitude, and for this reason we have chosen to
illustrate the numerical results for the case $n=7$ in this section. One of
the four Planar families that exist for $n=7$ also consists of relatively
small amplitude orbits. The other three Planar families are illustrated in
Figure~\ref{fig01}, where the panels on the left show an orbit along each of
three distinct Planar Lyapunov families. These orbits are well away from the
polygonal relative equilibrium from which the respective families originate,
while they are also still well away from the collision orbits which these
families appear to approach. The panels on the right in Figure~\ref{fig01}
show orbits along the three families that are further away from the relative
equilibria. Orbits along the Planar families for the cases $n=8$ and $n=9$
share many features with those for the case $n=7$.

Families of spatial orbits, which have nonzero $z$-component, emanate from the
polygonal relative equilibrium when $n\geq3$. These families and their orbits
are often referred to as \textquotedblleft Vertical\textquotedblright, because
the solution of the linearized Newton equations at the equilibrium is
perfectly vertical, \textit{i.e.}, the $x$- and $y$-components are identically
zero. For the case $n=3$ the Vertical Lyapunov family is highly degenerate, as
it corresponds to an eigenvalue of algebraic multiplicity $5$, and there are
no further eigenvalues that give rise to Vertical orbits. For the case $n=4$
there is an equally degenerate eigenvalue ($k=1$). However, there is also a
nondegenerate eigenvalue that gives rise to a Vertical family, namely the one
known as the \textquotedblleft Hip-Hop family\textquotedblright\ ($k=2$). The
top-left panel of Figure~\ref{fig02} shows orbits along this family, which
terminates in a collision orbit. The coloring of the orbits along the family
gradually changes from solid blue (near the equilibrium) to solid red (near
the terminating collision orbit). The same coloring scheme is used when
showing other entire families of orbits in rotating coordinates.

The top-right panel shows a single orbit from the Hip-Hop family, namely the
first bifurcation orbit encountered along it. The color of this orbit
gradually changes from blue to red as the orbit is traversed, so that one can
infer the direction of motion. The masses are shown at their ``initial''
positions. The same coloring scheme is used when showing other individual
orbits in rotating coordinates.

The center-left panel of Figure~\ref{fig02} shows the Axial family that
bifurcates from the Hip-Hop family. The name ``Axial'' alludes to the fact
that the orbits of this family are invariant under the transformation
$(-y,-z)$, when the $x$-axis is chosen to pass through the \textquotedblleft
center\textquotedblright of the orbit. The Axial family connects to a Planar
family, namely at the planar bifurcation orbit shown in the center-right panel
of Figure~\ref{fig02}. We refer to this Planar family as ``Unchained'',
because some of its orbits give rise to choreographies called ``Unchained
polygons'' in \cite{ChFe08}. The Hip-Hop family for $n=4$, and its bifurcating
families, are qualitatively similar to corresponding families that we have
computed for the cases $n=6$ and $n=8$.

The examples of orbit families given in this section are representative of the
many planar and spatial Lyapunov families that we have computed, their
secondary and tertiary bifurcating families, as well as corresponding families
for other values of $n$. Complete bifurcation pictures are rather complex, but
our algorithms are capable of attaining a high degree of detail; which at this
point excludes only the degenerate bifurcations mentioned earlier.

In the following sections we focus our attention on choreographies that arise
from resonant periodic orbits. The statements proved for the Lyapunov families
also hold true for subsequent spatial and planar bifurcations, as long as the
symmetries (\ref{PS}) and (\ref{SS}) are present. However this is not always
the case, and in Section 6 we give details on a Lyapunov family that does not
possess these symmetries.

Figure~\ref{fig03} illustrates the appearance of choreographies from resonant
Lyapunov orbits and from resonant orbits along subsequent bifurcating
families. Specifically, the top-left panel of Figure~\ref{fig03} shows a
resonant Planar Lyapunov orbit for the case $n=7$, and the top-right panel
shows the same orbit in the inertial frame, where it is seen to correspond to
a choreography. Similarly the center panels show a resonant spatial Lyapunov
orbit and corresponding choreography for $n=9$, while the bottom panels show a
resonant Axial orbit and corresponding choreography for $n=4$.
\newpage
\begin{figure}[th]
\par
\begin{center}
\resizebox{15.5cm}{!}{
\includegraphics{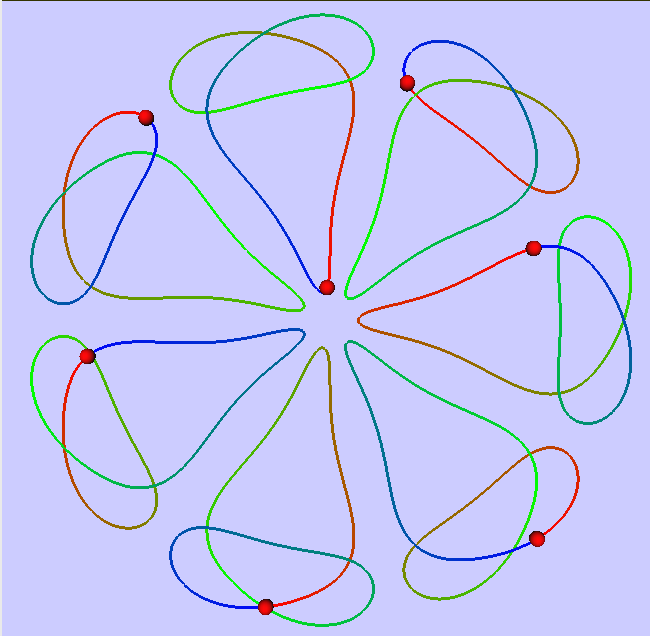} ~~~
\includegraphics{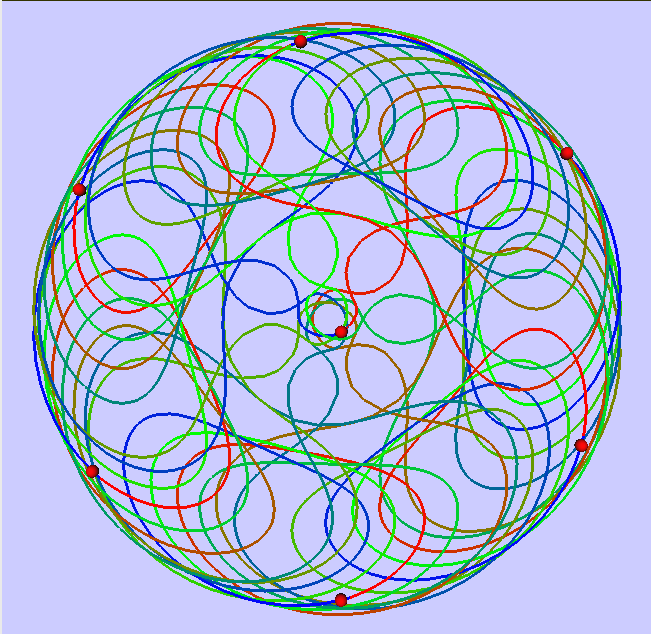}} \vskip0.15cm \resizebox{15.5cm}{!}{
\includegraphics{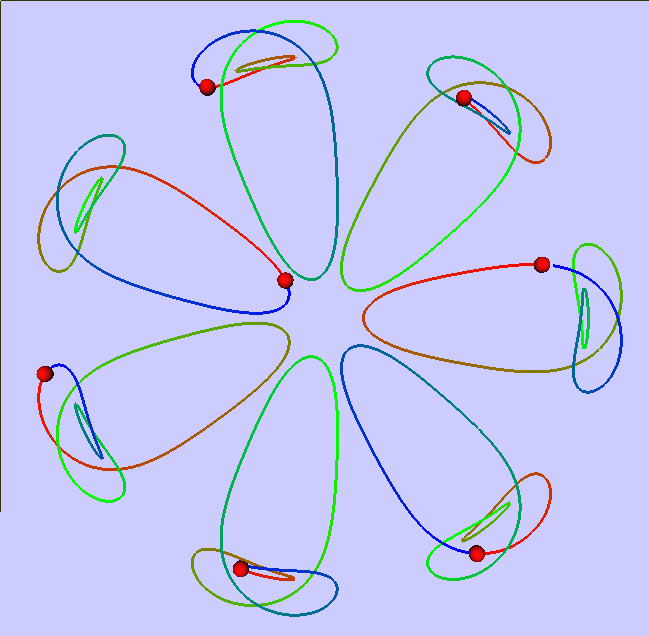} ~~~
\includegraphics{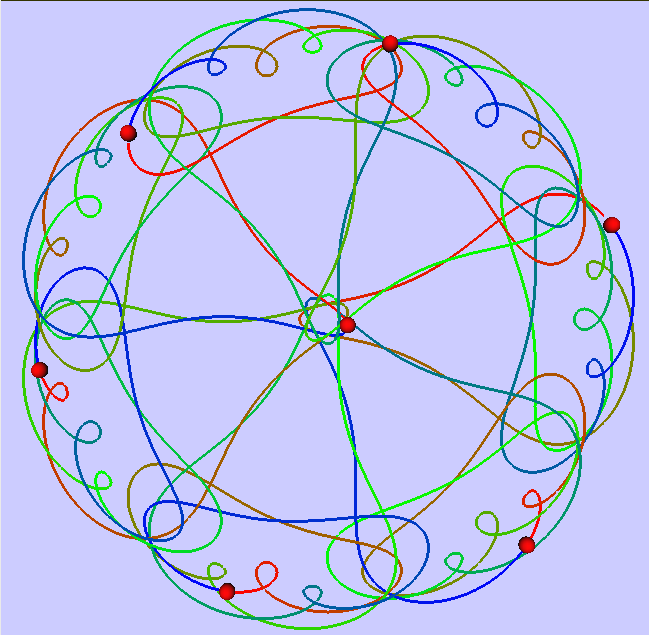}} \vskip0.15cm \resizebox{15.5cm}{!}{
\includegraphics{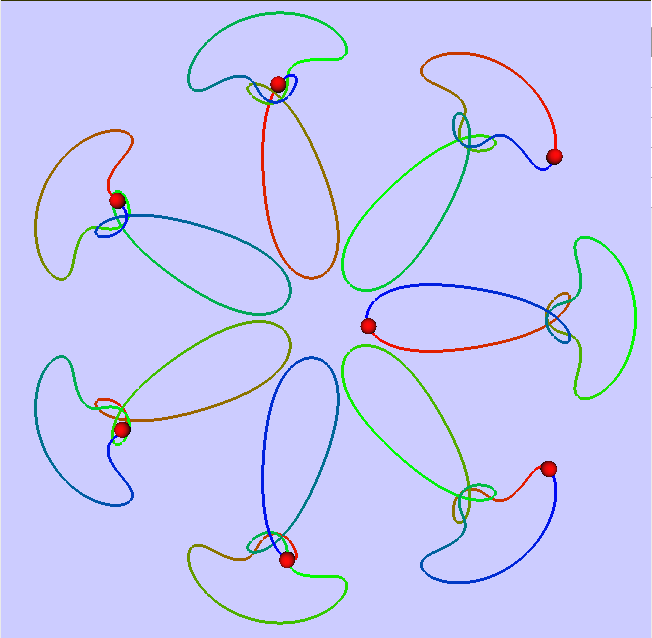} ~~~
\includegraphics{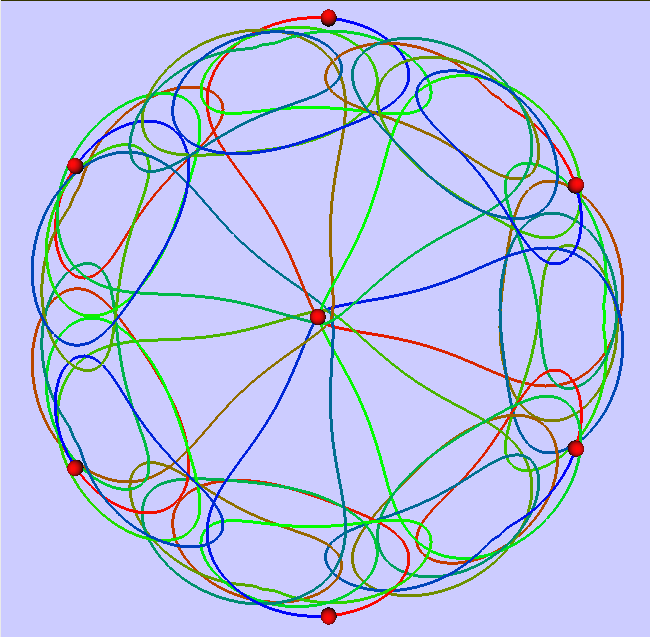}}
\end{center}
\caption{ Some orbits along Planar families for the case $n=7$. Top: two
orbits with $k=2$. Center: two orbits with $k=3$. Bottom: two orbits with
$k=4$. }%
\label{fig01}%
\end{figure}
\clearpage
\begin{figure}[th]
\par
\begin{center}
\resizebox{15.0cm}{!}{
\includegraphics{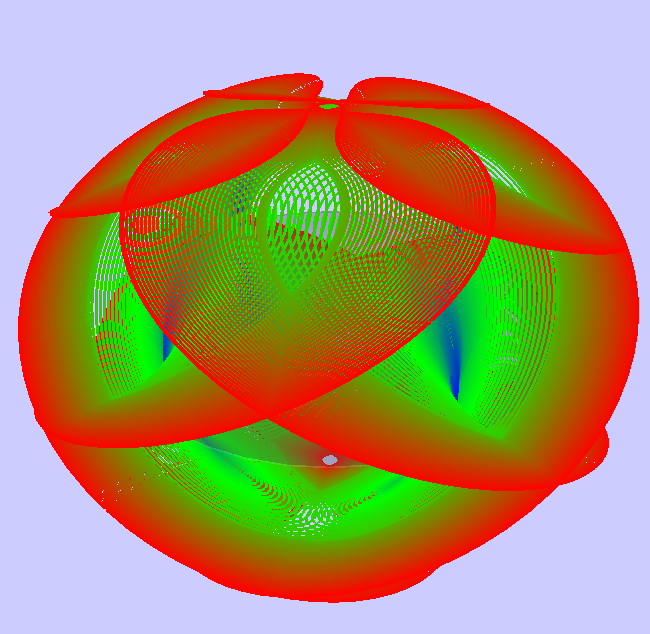} ~~~
\includegraphics{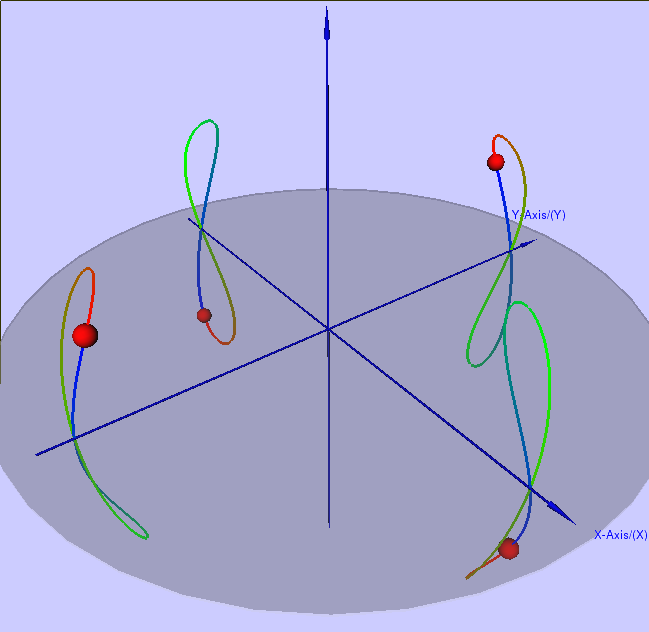}} \vskip0.15cm \resizebox{15.0cm}{!}{
\includegraphics{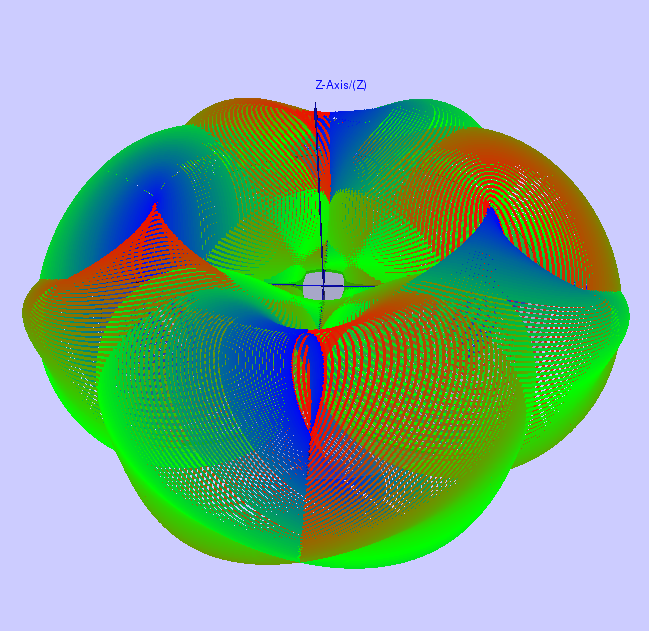} ~~~
\includegraphics{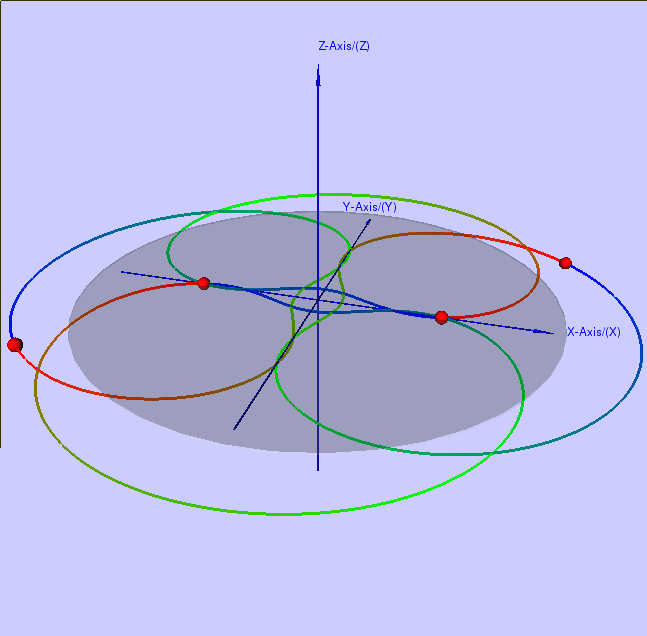}} \vskip0.15cm \resizebox{15.0cm}{!}{
\includegraphics{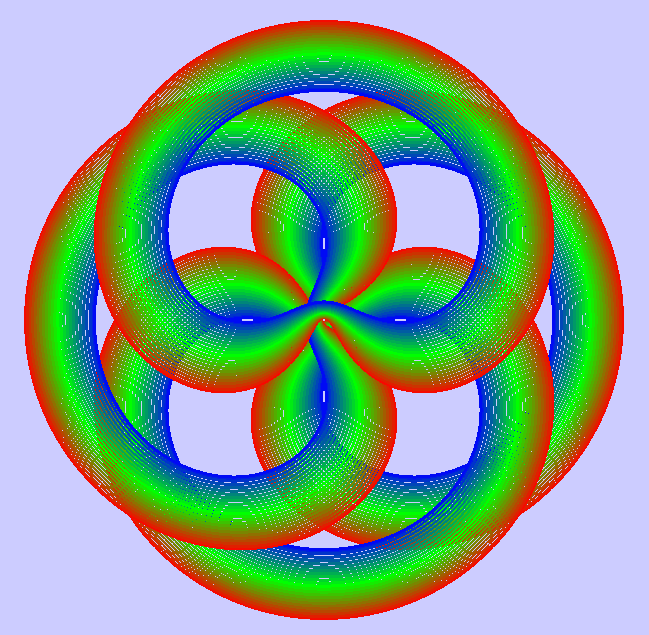} ~~~
\includegraphics{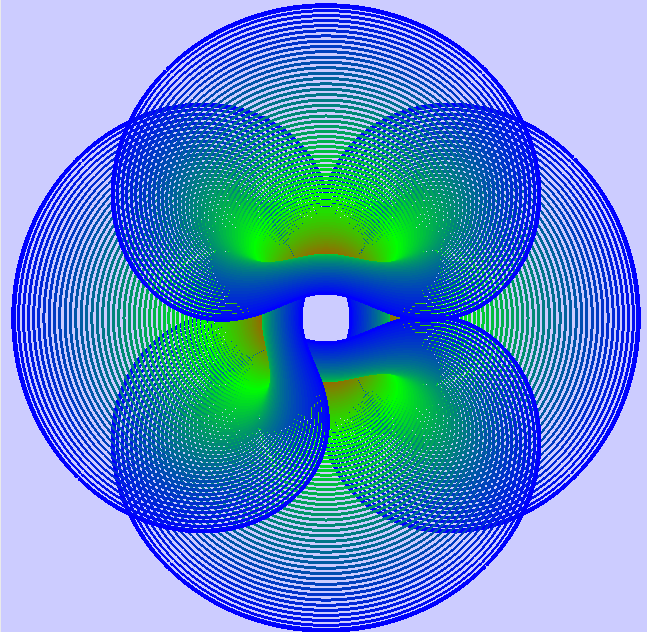}}
\end{center}
\caption{ Top-Left: the Vertical Lyapunov family for $n=4$ and $k=2$.
Top-Right: the first bifurcation orbit along the Vertical family. Center-Left:
the Axial family that bifurcates from the Vertical family. Center-Right: the
bifurcation orbit where the Axial family connects to to a Planar family.
Bottom-Left: one branch of the Planar family to which the Axial family
connects. Bottom-Right: the other branch of the family to which the Axial
family connects. }%
\label{fig02}%
\end{figure}
\clearpage
\begin{figure}[th]
\par
\begin{center}
\resizebox{15.35cm}{!}{
\includegraphics{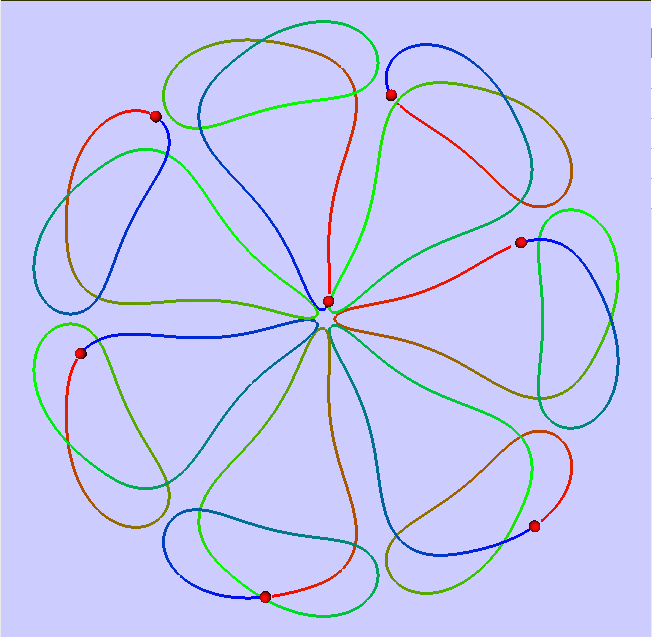} ~~~
\includegraphics{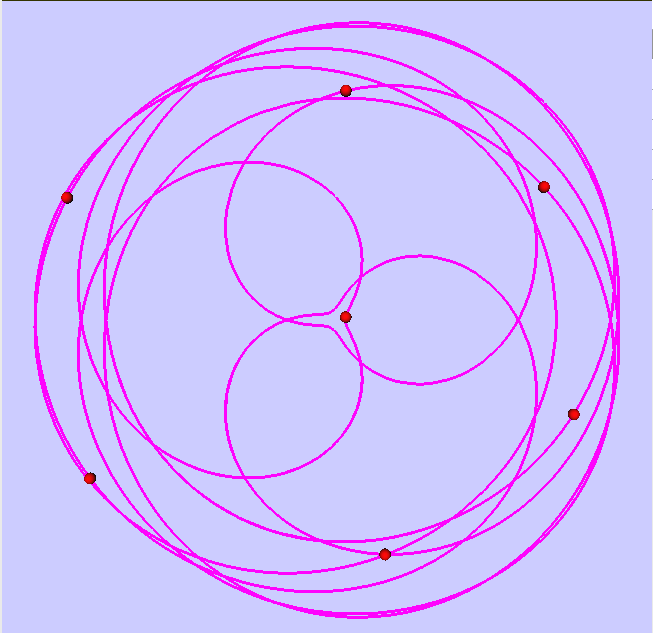}} \vskip0.15cm
\resizebox{15.35cm}{!}{
\includegraphics{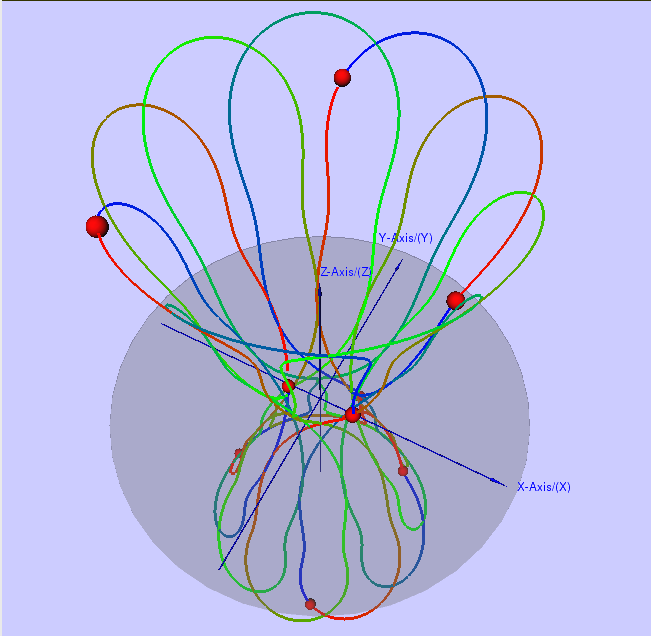} ~~~
\includegraphics{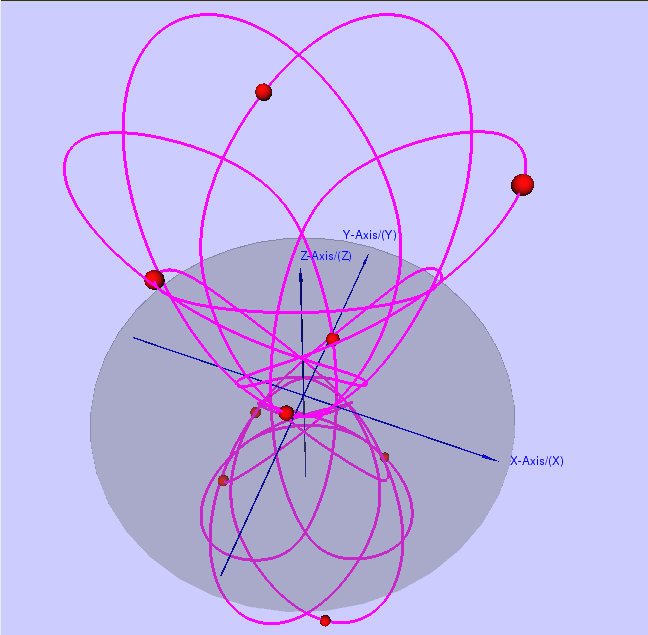}} \vskip0.15cm
\resizebox{15.35cm}{!}{
\includegraphics{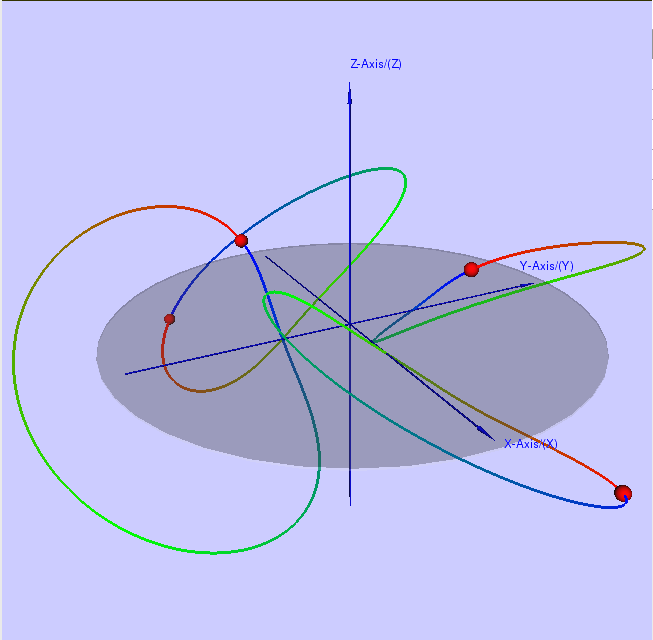} ~~~
\includegraphics{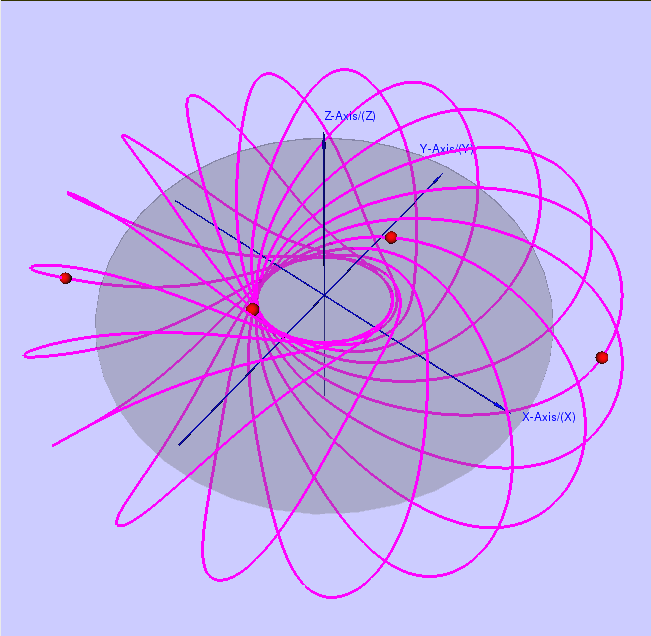}}
\end{center}
\caption{The panels on the left show orbits in the rotating frame, while the
panels on the right show the same orbits in the inertial frame, where they
correspond to choreographies. Top: a resonant Planar Lyapunov orbit. Center: a
resonant Vertical Lyapunov orbit. Bottom: a resonant Axial orbit. }%
\label{fig03}%
\end{figure}
\clearpage

\section{Choreographies along Planar Lyapunov families}

In this section we present some of the infinitely many choreographies that
appear along the Planar Lyapunov families, namely for the cases $n=7$, $n=8$
and $n=9$, as shown in Figures~\ref{fig04},~\ref{fig05}, and \ref{fig06},
respectively. Corresponding data are given in Tables~1~-~3. Each choreography
winds $\ell$ times around a center and is invariant under rotations of
$2\pi/m$. The bodies move in groups of $d$-polygons, where $d$ is the greatest
common divisor of $n$ and $k$. In addition these choreographies are symmetric
with respect to reflection in the plane generated by the second symmetry in
(\ref{PS}).

When there is an infinite number of choreographies then the winding number
$\ell$ and the symmetry indicator $m$ can be arbitrarily large, and the
choreography arbitrarily complex. From the observed range of values of the
periods along a Lyapunov family we mostly choose the simpler resonances, and
hence the simpler choreographies. For example, the family $k=2$ for $n=7$ has
a relatively simple choreography. Here $\ell=5$ and $m=3$ are relatively
prime, with
\[
k\ell-m=2\times5-3=7\in n\mathbb{Z}\text{.}%
\]
In this example $2\pi s_{1}^{-1/2}=4.1387$, and since $T_{5:3}=(2\pi
s_{1}^{-1/2})(5/3)=\allowbreak6.8978$ is within the range of periods of the
Lyapunov family, it follows that the $5:3$ resonant Lyapunov orbit corresponds
to a choreography in the inertial frame. This choreography is shown in the
center-right panel of Figure~\ref{fig04}. It has period $3T_{5:3}$, winding
number $5$, and it is invariant under rotations of $2\pi/3$. Similar
statements apply to other planar choreographies.

\noindent For $n=7$ there is a sequence of Planar Lyapunov families having
$k=2,3,4,2$, respectively. The last of these families, with $k=2$, has orbits
of rather small amplitude, and is not included in the families shown in
Figure~\ref{fig04}.%

\[%
\begin{tabular}
[c]{|l|l|l|c|}\hline
$k$ & Eigenvalue & Period Interval & Resonant Orbit\\\hline
2 & 1.53960$i$ & [4.0811, 28.328] & 4:1 and 5:3\\\hline
3 & 1.85058$i$ & [3.3953, 27.974] & 3:2 and 5:1\\\hline
4 & 1.50806$i$ & [4.1664, 28.499] & 2:1 and 15:4\\\hline
2 & 0.761477$i$ & [8.2513 ,8.3328] & --\\\hline
\end{tabular}
\]
\centerline{Table~1: Data for $n=7$ bodies.}

\noindent For $n=8$ there is a sequence of Planar Lyapunov family having
$k=2,3,4,5,2$, respectively. We have chosen one choreography from each one of
these families, with an additional one for $k=5$, in Figure~\ref{fig05}.%

\[%
\begin{tabular}
[c]{|l|l|l|c|}\hline
$k$ & Eigenvalue & Period Interval & Resonant Orbit\\\hline
2 & 1.94947$i$ & [3.2230, 14.836] & 11:6\\\hline
3 & 2.39714$i$ & [2.6211, 29.654] & 11:9\\\hline
4 & 2.41171$i$ & [2.6053, 7.4935] & 5:4\\\hline
5 & 1.91468$i$ & [3.2814, 29.122] & 5:1 and 9:5\\\hline
2 & 0.435437$i$ & [5.4804, 14.430] & 5:2\\\hline
\end{tabular}
\]
\centerline{Table~2: Data for $n=8$ bodies.}

\noindent For $n=9$ there is a sequence of Planar Lyapunov families having
$k=2,3,4,5,6,2$, respectively. In Figure~\ref{fig06} we have selected one
choreography from each of these families.%

\[%
\begin{tabular}
[c]{|l|l|l|c|}\hline
$k$ & Eigenvalue & Period Interval & Resonant Orbit\\\hline
2 & 2.27175$i$ & [2.7660, 30] & 5:1\\\hline
3 & 2.85442$i$ & [2.2012, 10.298] & 4:3\\\hline
4 & 3.06012$i$ & [2.0534, 30.411] & 5:2\\\hline
5 & 2.90713$i$ & [2.1613, 30.612] & 2:1\\\hline
6 & 2.26399$i$ & [2.7197, 10.008] & 5:3\\\hline
2 & 0.196565$i$ & [15.400, 31.927] & 5:1\\\hline
\end{tabular}
\]
\centerline{Table~3: Data for $n=9$ bodies.}

\clearpage
\begin{figure}[h]
\par
\begin{center}
\resizebox{15.3cm}{!}{
\includegraphics{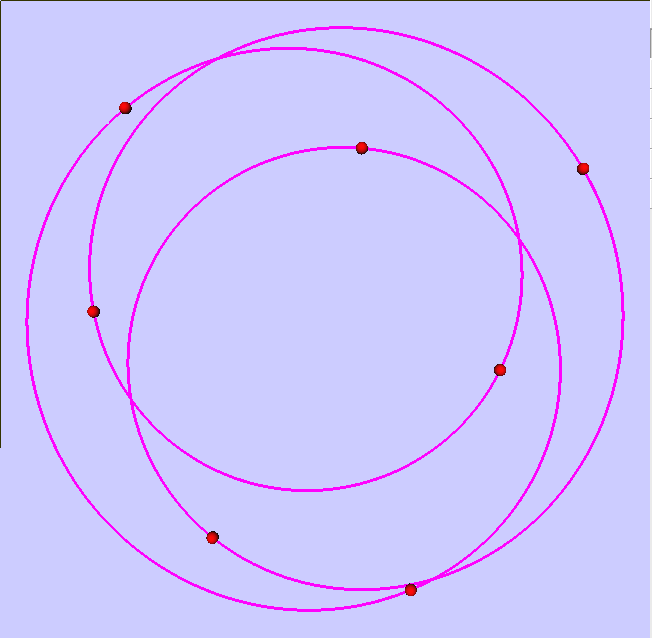} ~~~
\includegraphics{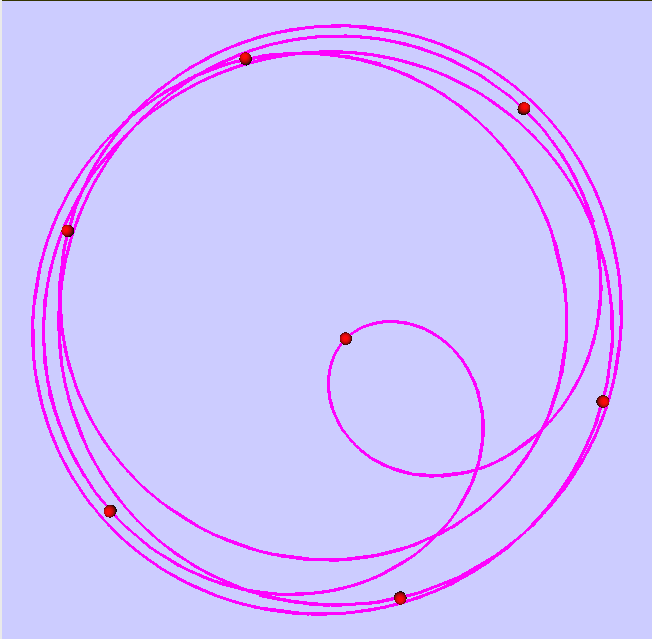}} \vskip0.15cm
\resizebox{15.3cm}{!}{
\includegraphics{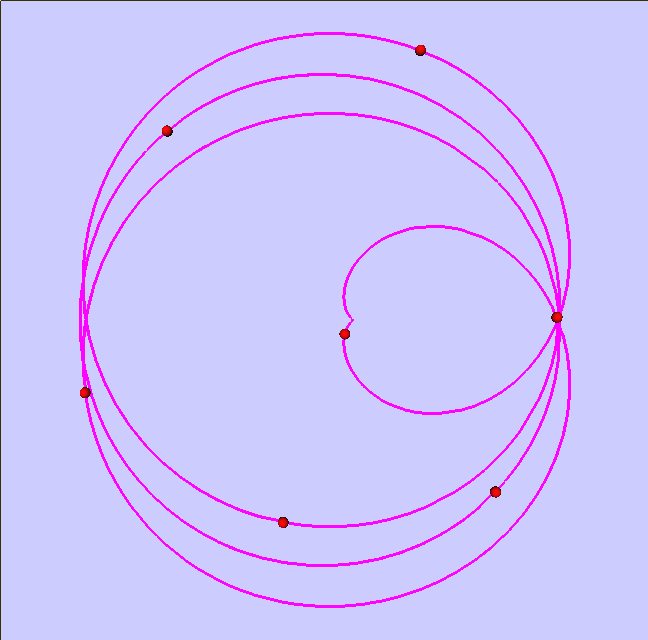} ~~~
\includegraphics{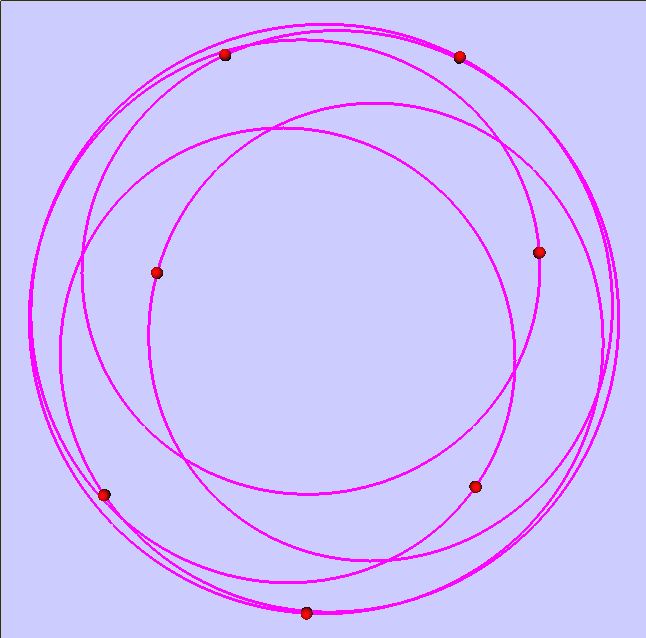}} \vskip0.15cm
\resizebox{15.3cm}{!}{
\includegraphics{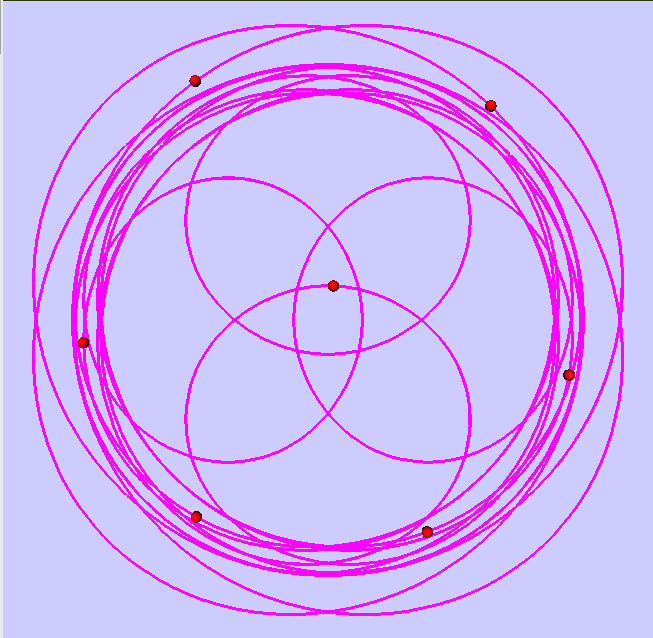} ~~~
\includegraphics{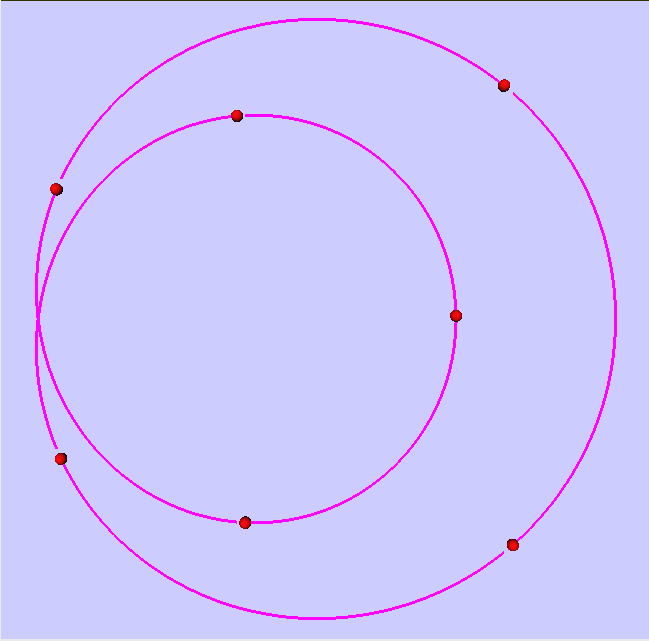}}
\end{center}
\caption{ Planar choreographies for $n=7$ bodies: Top-Left: a $3:2$-resonant
orbit for $k=3$. Top-Right: a $5:1$-resonant orbit for $k=3$. Center-Left: a
$4:1 $-resonant orbit for $k=2$. Center-Right: a $5:3$-resonant orbit for
$k=2$. Bottom-Left: a $15:4$-resonant for orbit for $k=4$. Bottom-Right: a
$2:1$-resonant orbit for $k=4$. }%
\label{fig04}%
\end{figure}
\clearpage
\begin{figure}[h]
\par
\begin{center}
\resizebox{15.35cm}{!}{
\includegraphics{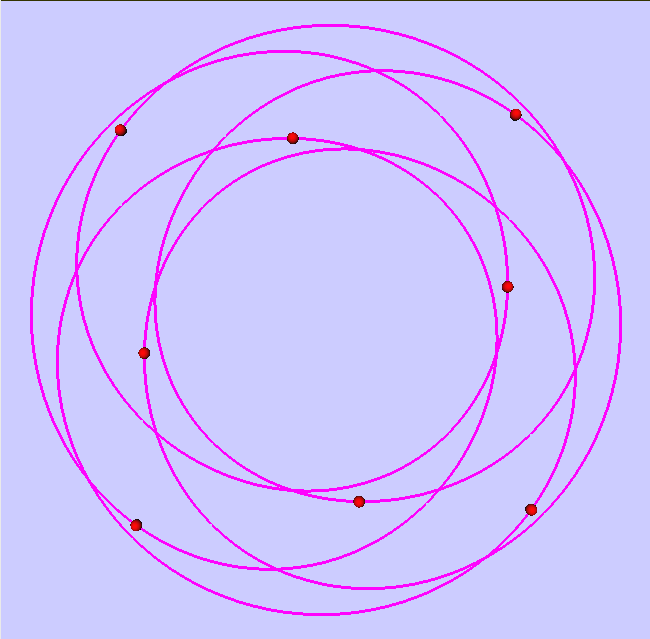} \qquad
\includegraphics{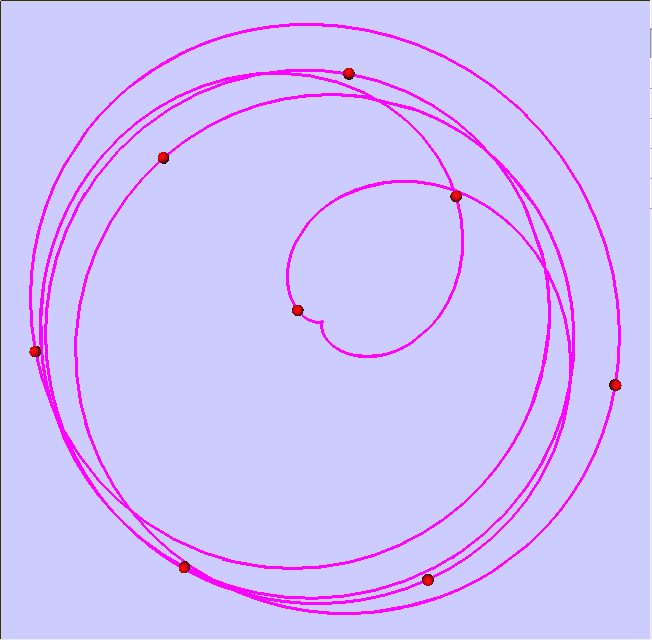}} \vskip0.15cm
\resizebox{15.35cm}{!}{
\includegraphics{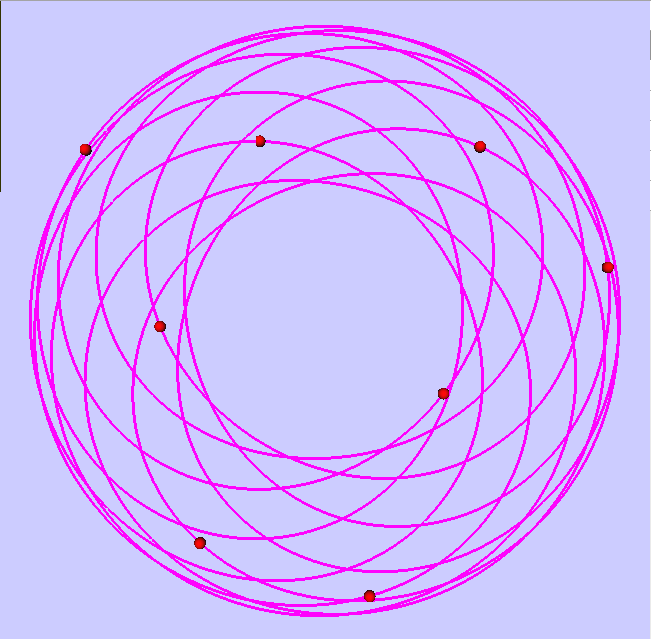} \qquad
\includegraphics{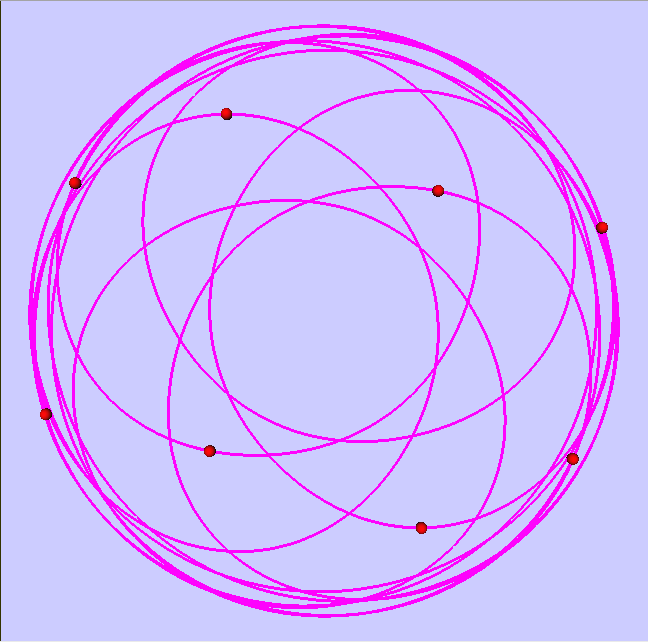}} \vskip0.15cm
\resizebox{15.35cm}{!}{
\includegraphics{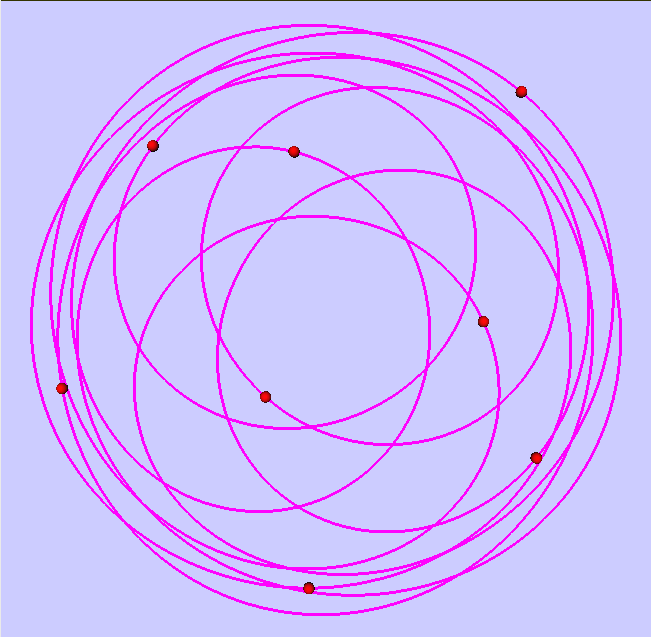} \qquad
\includegraphics{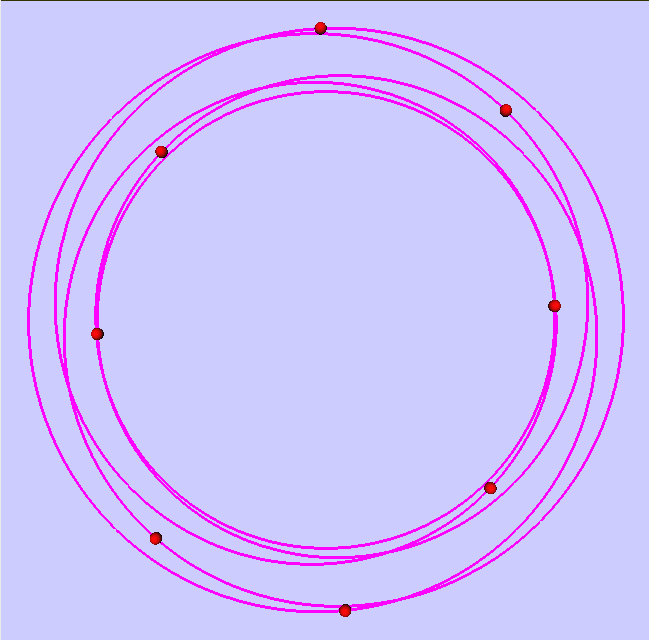}}
\end{center}
\caption{ Planar choreographies for $n=8$. Top-Left: a $5:4$-resonant orbit
for $k=4$. Top-Right: a $5:1$-resonant orbit for $k=5$. Center-Left: an
$11:9$-resonant orbit for $k=3$. Center-Right: an $11:6$-resonant orbit for
$k=2$. Bottom-Left: a $9:5$-resonant orbit for $k=5$. Bottom-Right: a
$5:2$-resonant orbit for $k=2$. }%
\label{fig05}%
\end{figure}
\clearpage
\begin{figure}[h]
\par
\begin{center}
\resizebox{15.3cm}{!}{
\includegraphics{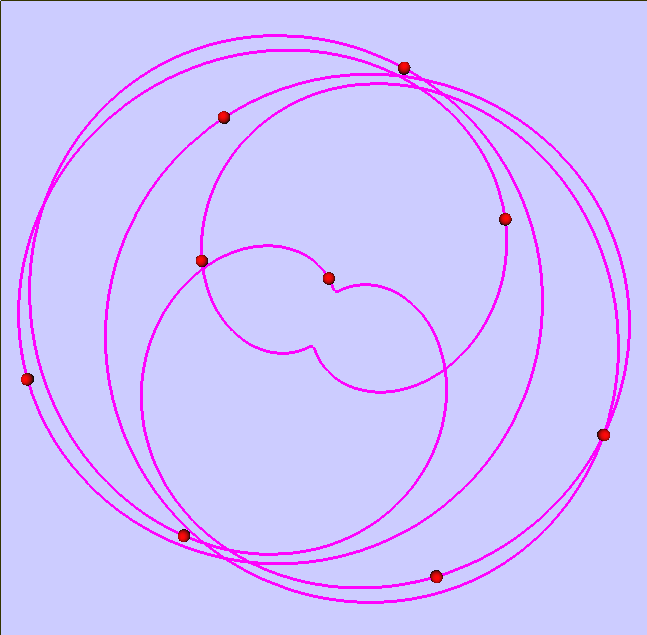} ~~~
\includegraphics{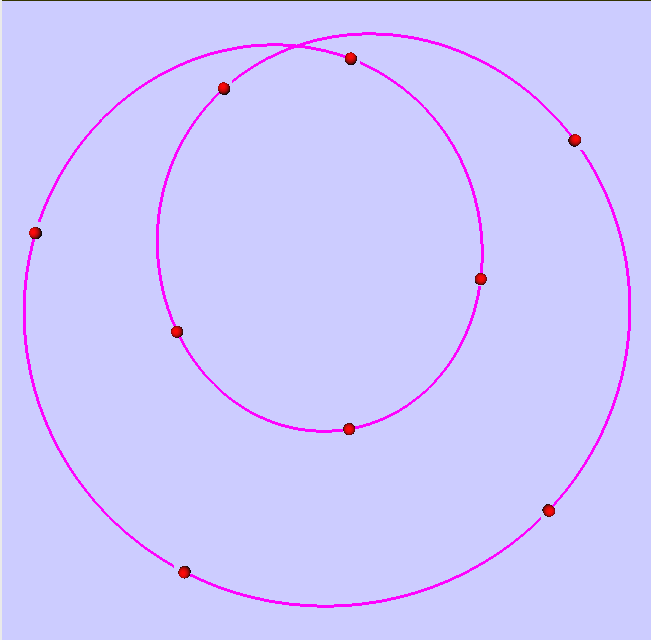}} \vskip0.15cm
\resizebox{15.3cm}{!}{
\includegraphics{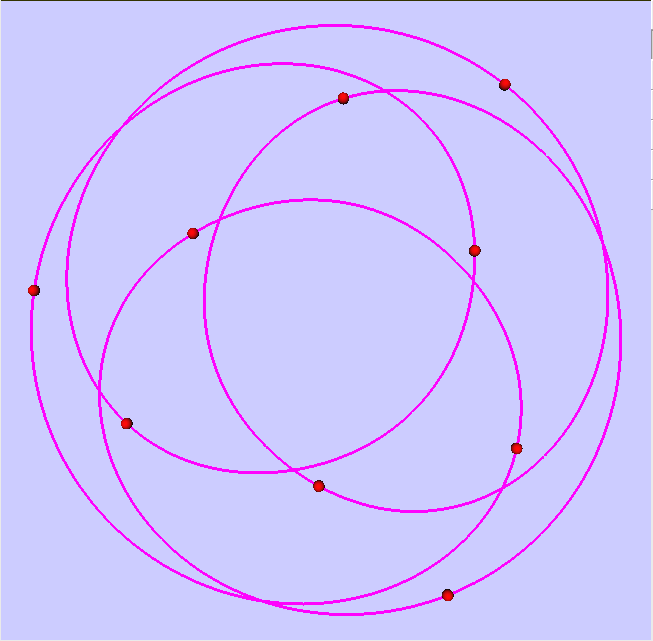} ~~~
\includegraphics{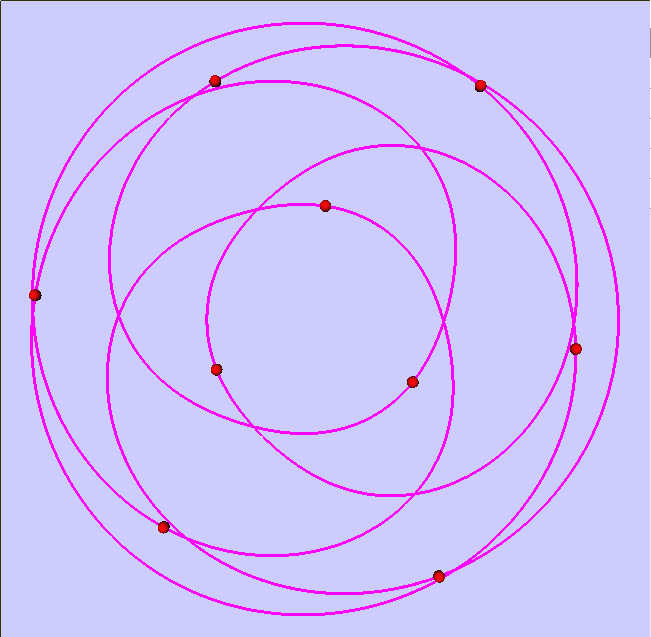}} \vskip0.15cm
\resizebox{15.3cm}{!}{
\includegraphics{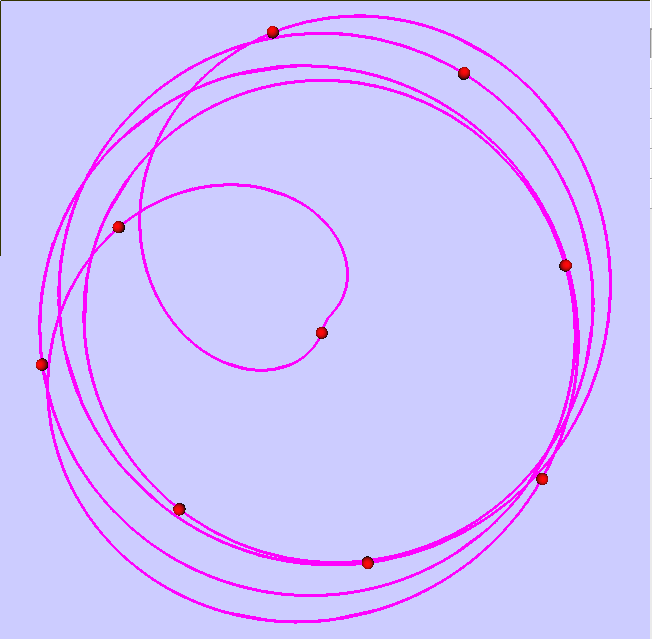} ~~~
\includegraphics{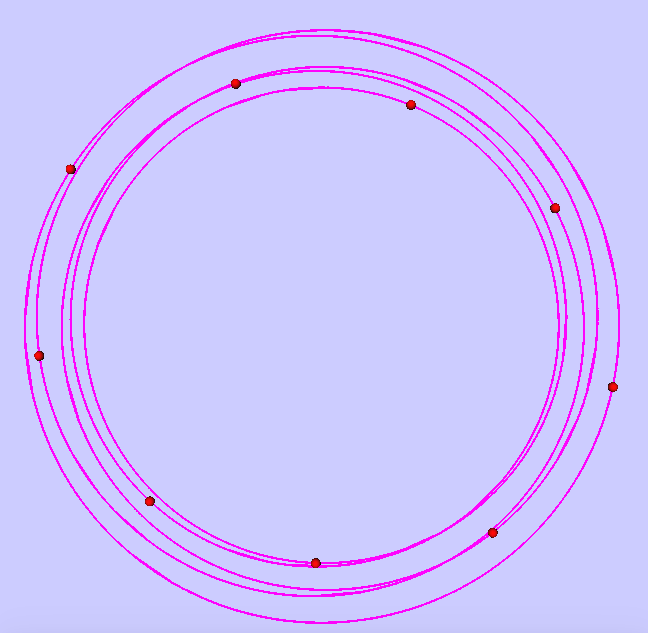}}
\end{center}
\caption{ Planar choreographies for $n=9$. Top-Left: a $5:2$-resonant orbit
for $k=4$. Top-Right: a $2:1$-resonant orbit for $k=5$. Center-Left: a
$4:3$-resonant orbit for $k=3$. Center-Right: a $5:3$-resonant orbit for
$k=6$. Bottom-Left: a $5:1$-resonant orbit for $k=2$. Bottom-Right: a
$5:1$-resonant orbit for $k=2$. }%
\label{fig06}%
\end{figure}
\clearpage

\section{Choreographies along Vertical Lyapunov families and their bifurcating
families}

In this section we give examples of choreographies along the Vertical Lyapunov
families and along their bifurcating families. The projections of these
choreographies onto the $xy$-plane are invariant under rotations of $2\pi/m$.
The bodies form groups of $d$-polygons, where $d$ is the greatest common
divisor of $n$ and $k$. Since we obtain choreographies by rotating closed
orbits, each choreography is contained in a surface of revolution. Indeed, due
to the symmetries of the Vertical families the choreographies wind around a
cylindrical manifold with winding number $\ell$, while for the Axial families
the choreographies wind around a toroidal manifold with winding numbers $\ell$
and $m$.

The spatial choreographies along the Vertical Lyapunov families are symmetric
with respect to the reflections $-y$ and $-z$, when the $x$-axis is chosen to
pass through the \textquotedblleft center\textquotedblright of the orbit.
While planar choreographies for large values of $\ell$ and $m$ are somewhat
difficult to appreciate, spatial choreographies of this type are easier to
visualize because they wind around a cylindrical manifold. For even values of
$n$ we mention the case $k=n/2$, for which the orbits of the Vertical family
are known as Hip-Hop orbits. Choreographies along such families have been
described before in \cite{TeVe07}, and in \cite{ChFe08}, where they were found
numerically as local minimizers of the action restricted to symmetric paths.
Along Hip-Hop families we have located the choreography for $n=4$ found in
\cite{TeVe07}. Several choreographies along Hip-Hop families are shown the top
four panels of Figure~\ref{fig07}. We have not computed all Vertical families,
due to presence of double resonant eigenvalues. For the case $n=9$, we show
two choreographies along a family of periodic orbits that is not a Hip-Hop
family, namely in the two bottom panels of Figure~\ref{fig07}. Such families
were not determined in \cite{ChFe08} because they do not correspond to local
minimizers of the action. Further investigation is needed for a systematic
approach to determine these families.

We now present some choreographies along the families that emanate from the
first bifurcation along Hip-Hop families in the rotating frame. The
projections of these spatial choreographies onto the $xy$-plane are somewhat
similar to those along the Planar Lyapunov orbits. However, the spatial
periodic orbits in the rotating frame that correspond to these choreographies
have only one symmetry, which is given by the transformation $(-y,-z)$ when
the $x$-axis is chosen to pass through the \textquotedblleft
center\textquotedblright of the orbit in the rotating frame; see the bottom
left panel of Figure~\ref{fig03}. This is due to the fact that the Axial
family arises from the Vertical Lyapunov family via a symmetry-breaking
bifurcation. The symmetry implies that choreographies along the Axial families
wind around a toroidal manifold with winding numbers $\ell$ and $m$. Since we
assume that $\ell$ and $m$ are co-prime, the choreography path is known as a
\textit{torus knot}. The simplest nontrivial example is the $(2,3)$-torus
knot, also known as the \textit{trefoil knot}. We note that for other integers
$\ell$ and $m$ such that $k\ell-m\notin n\mathbb{Z}$, the orbit of the $n$
bodies in the inertial frame consists of separate curves that form a torus
link. Some of the choreographies along the Axial families are shown in
Figure~\ref{fig08}.

In Section 3 we already mentioned that there are planar bifurcation orbits
along Axial families that give rise to planar families. Such an Axial family
and its planar bifurcation orbit are shown in the center panels of
Figure~\ref{fig02}, namely for the case $n=4$. Orbits along the two branches
of the bifurcating planar family are shown in the bottom panels. Specifically,
our numerical computations indicate that Hip-Hop families connect indirectly
to planar families via the above-described tertiary bifurcation.
Choreographies along such planar families have symmetries that are similar to
those of Planar Lyapunov families, although in fact these families do not
correspond to Lyapunov families. While there are no Planar Lyapunov families
for $n=4, 5$, and $6$, there are such tertiary planar families for these
values of $n$, and these contain planar choreographies. Such choreographies
are called unchained polygons in \cite{ChFe08}, and there are infinitely many
of these. In particular, the Vertical family for $n=3$\ and $k=1$ leads
indirectly to the planar $P_{12}$-family of Marchal \cite{Ma00}. We have
continued such families numerically for $k=n/2$, where $n=4, 6$, and $8$, and
six choreographies along them are shown in the panels of Figure~\ref{fig09}.
\clearpage
\begin{figure}[h]
\par
\begin{center}
\resizebox{15.2cm}{!}{
\includegraphics{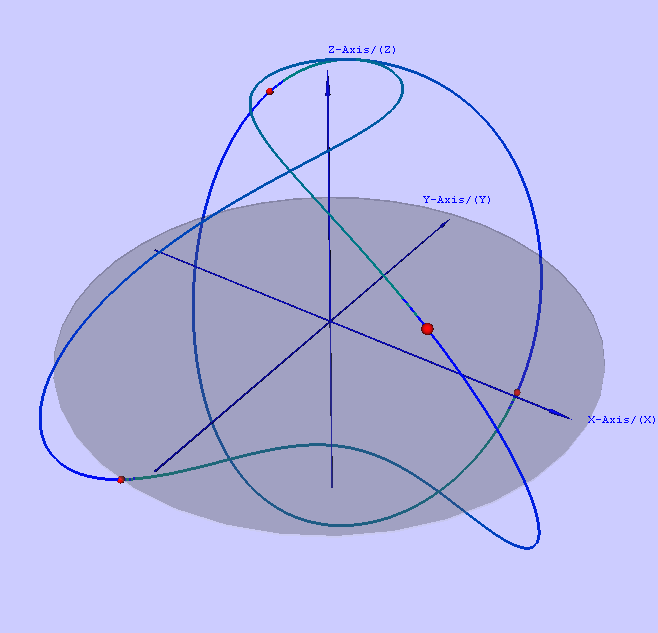} ~~~
\includegraphics{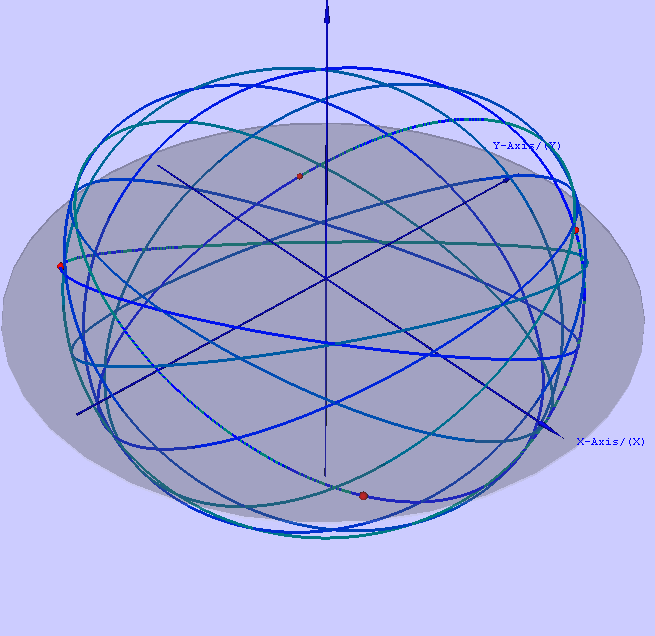}} \vskip0.15cm
\resizebox{15.2cm}{!}{
\includegraphics{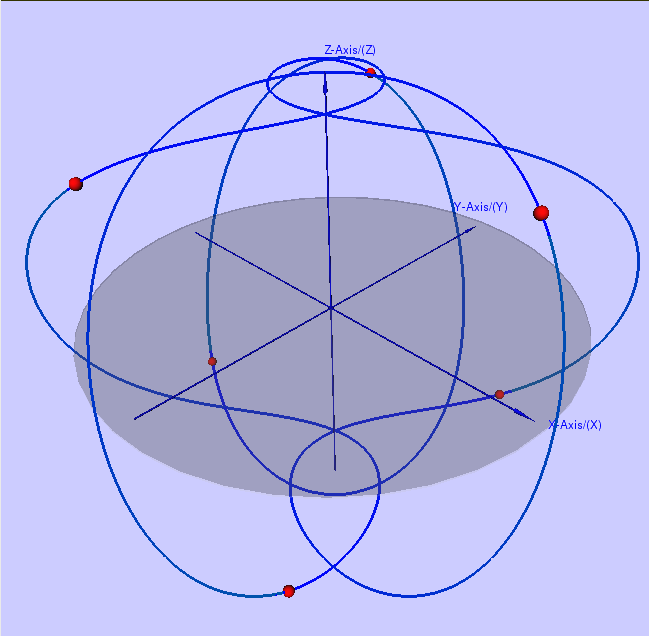} ~~~
\includegraphics{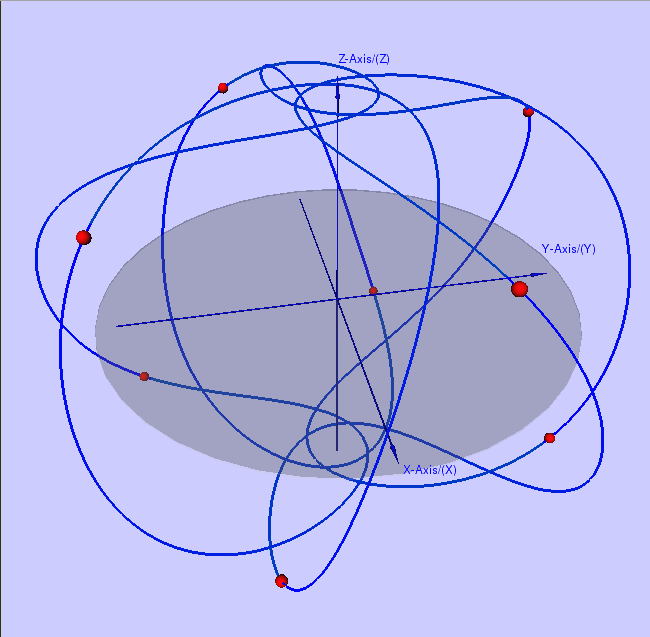}} \vskip0.15cm
\resizebox{15.2cm}{!}{
\includegraphics{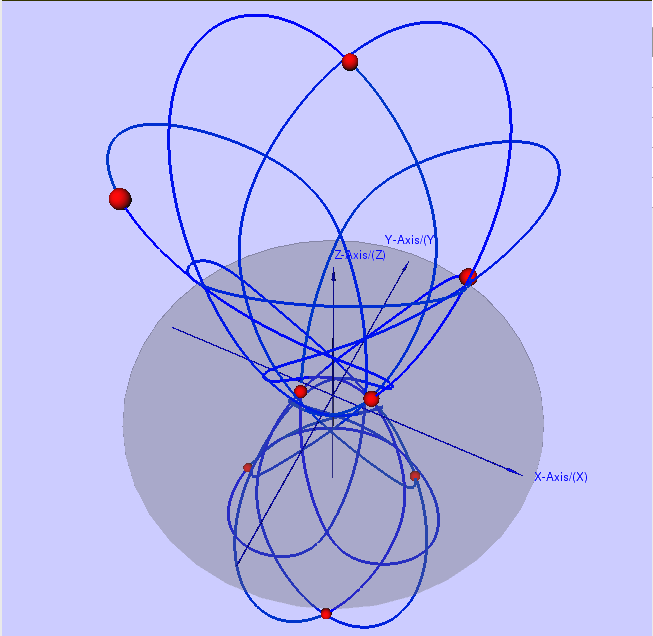} ~~~
\includegraphics{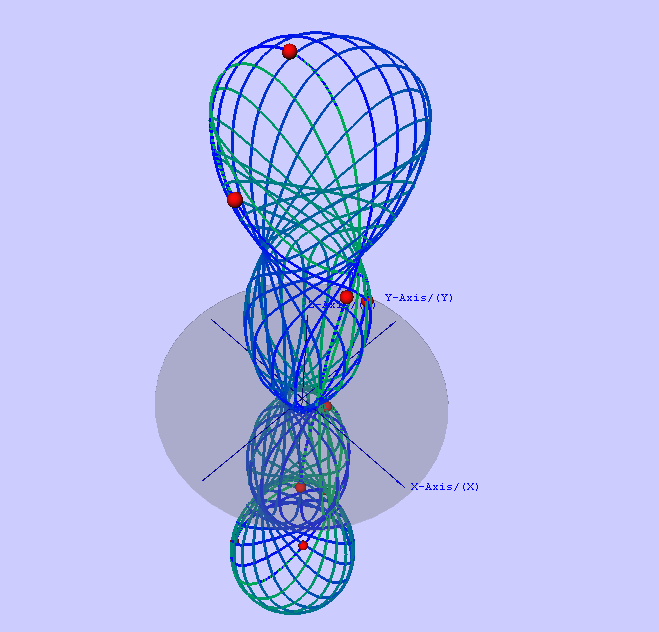}}
\end{center}
\caption{ Vertical Lyapunov families. Top-Left: a $3:2$-resonant Hip-Hop orbit
along $V_{1}$ for $n=4$. Top-Right: a $9:10$-resonant Hip-Hop orbit along
$V_{1}$ for $n=4$. Center-Left: a $5:3$-resonant Hip-Hop along $V_{1}$ for
$n=6$. Center-Right: a $7:4$-resonant Hip-Hop along $V_{1}$ for $n=8$.
Bottom-Left: an $11:5$ resonance along $V_{3}$ for $n=9$. Bottom-Right: a
$17:5$ resonance along $V_{3}$ for $n=9$. }%
\label{fig07}%
\end{figure}
\clearpage
\begin{figure}[h]
\par
\begin{center}
\resizebox{15.2cm}{!}{
\includegraphics{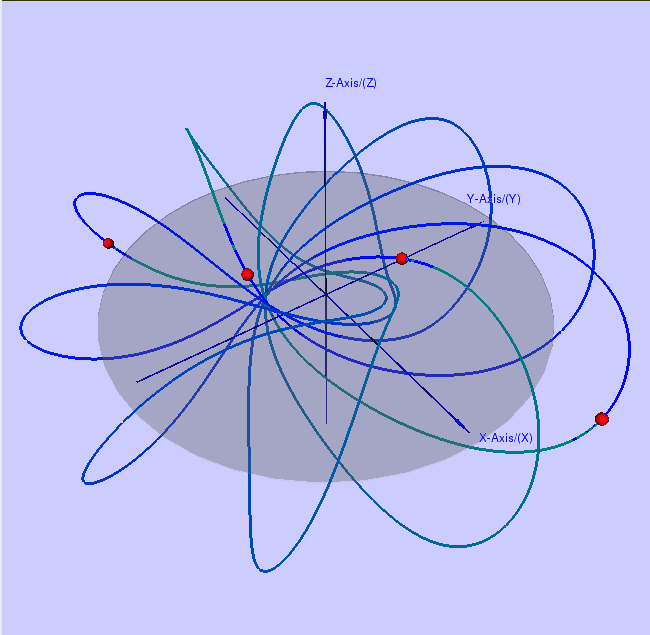} ~~~
\includegraphics{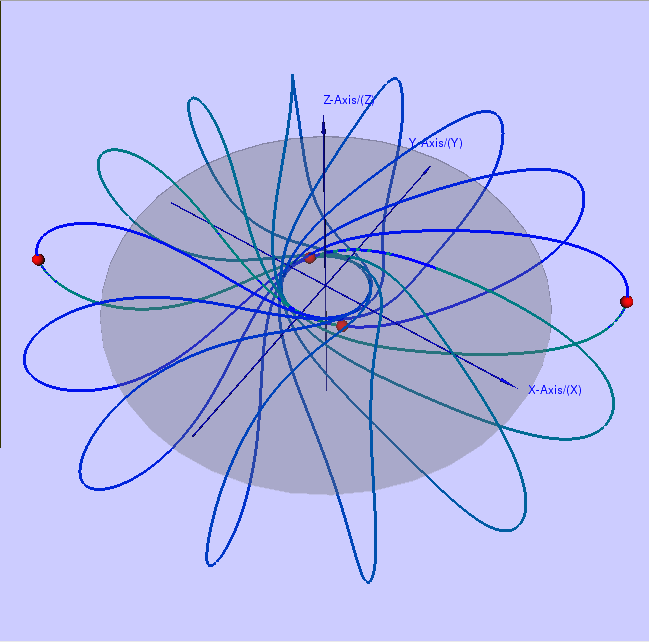}} \vskip0.15cm
\resizebox{15.2cm}{!}{
\includegraphics{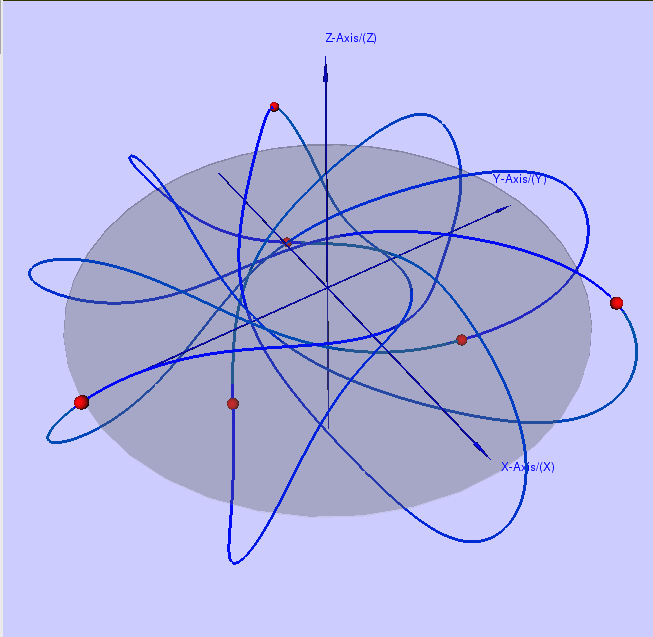} ~~~
\includegraphics{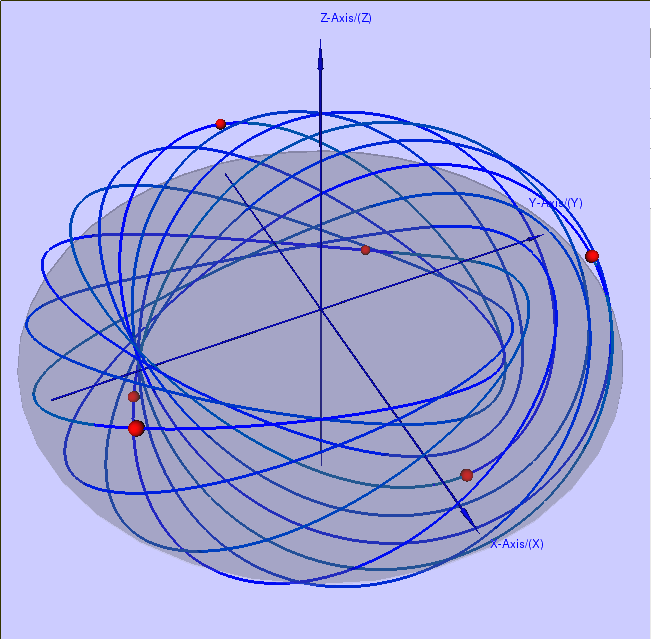}} \vskip0.15cm
\resizebox{15.2cm}{!}{
\includegraphics{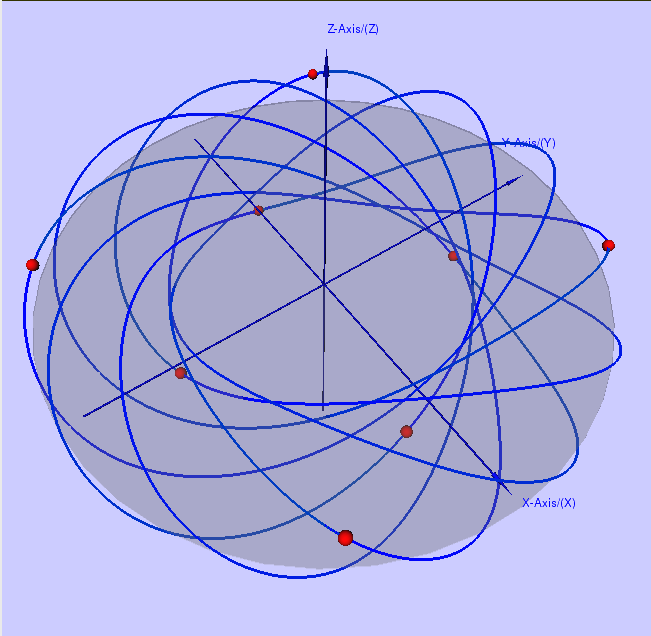} ~~~
\includegraphics{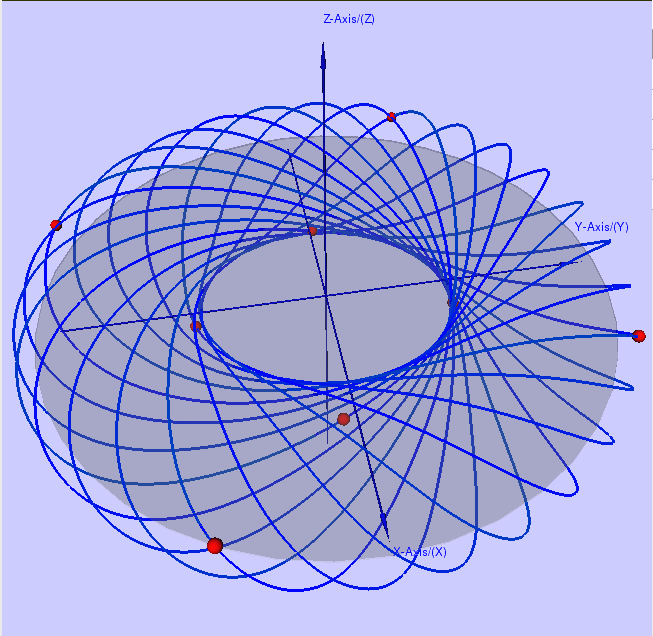}}
\end{center}
\caption{ Resonant Axial orbits with $k=n/2$. Top-Left: a $7:10$ resonant
Axial orbit for $n=4$. Top-Right: a $9:14$ resonant Axial orbit for $n=4$.
Center-Left: a $5:9$ resonant Axial orbit for $n=6$. Center-Right: an $11:15$
resonant Axial orbit for $n=6$. Bottom-Left: a $7:12$ resonant orbit for
$n=8$. Bottom-Right: a $15:28$ resonant Axial orbit for $n=8$. }%
\label{fig08}%
\end{figure}
\clearpage
\begin{figure}[h]
\par
\begin{center}
\resizebox{15.2cm}{!}{
\includegraphics{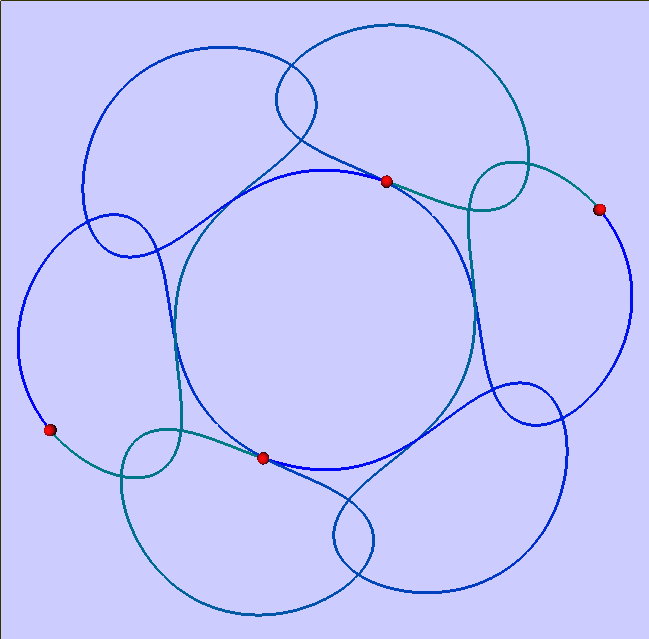} ~~~
\includegraphics{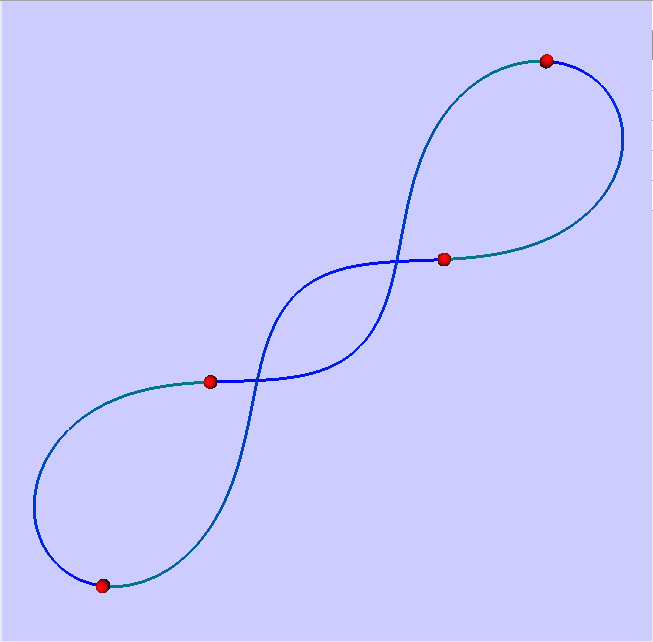}} \vskip0.15cm
\resizebox{15.2cm}{!}{
\includegraphics{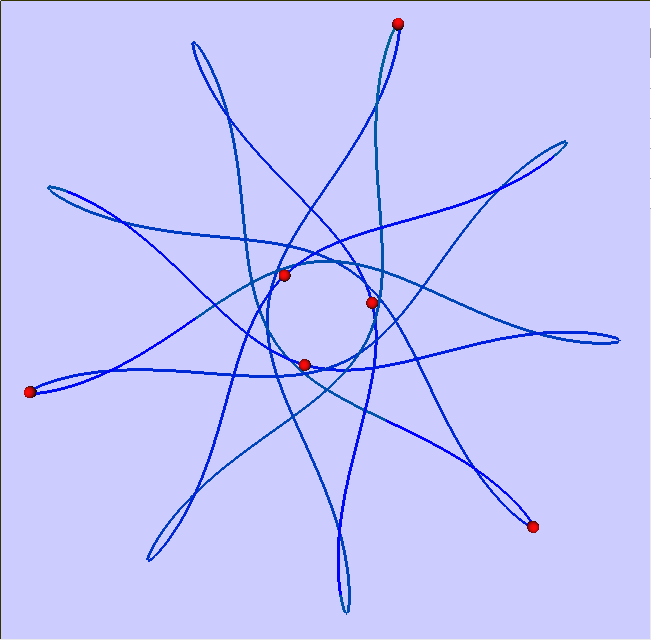} ~~~
\includegraphics{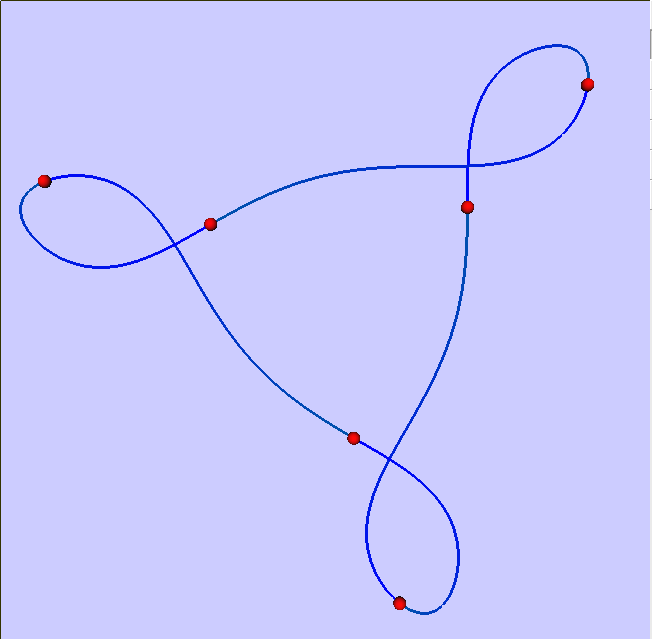}} \vskip0.15cm
\resizebox{15.2cm}{!}{
\includegraphics{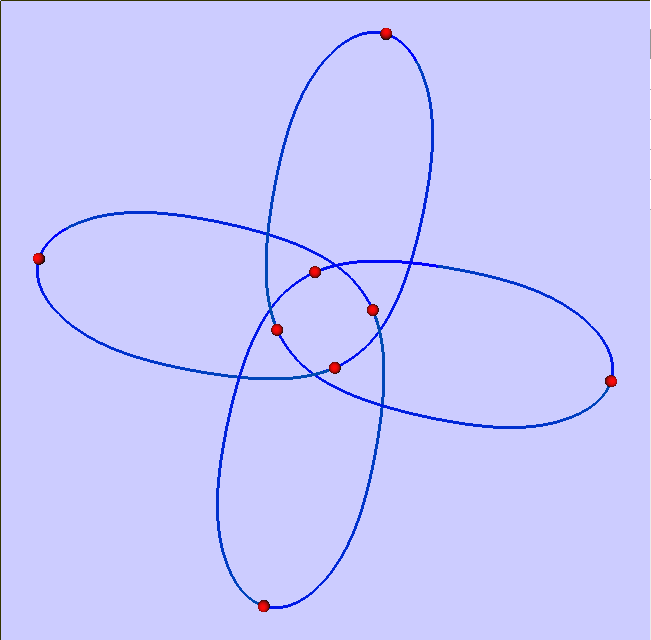} ~~~
\includegraphics{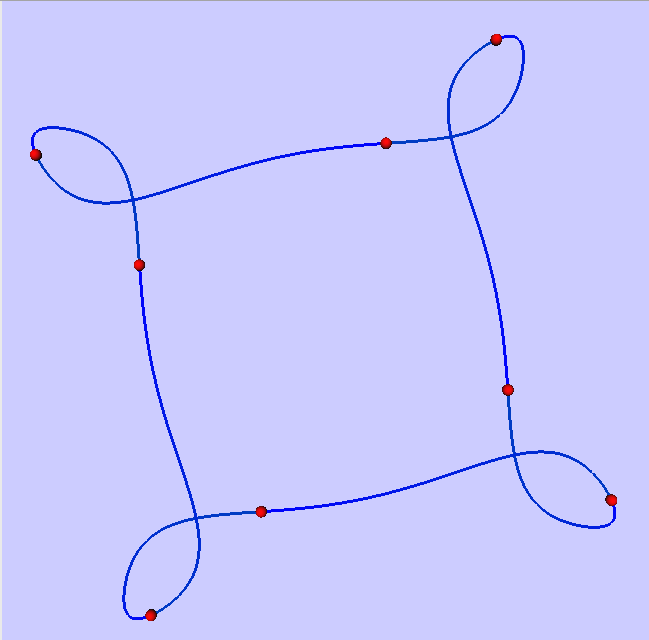}}
\end{center}
\caption{ Unchained polygons for $k=n/2$. Top-Left: a $1:6$ resonant orbit for
$n=4$. Top-Right: a $1:2$ resonant orbit for $n=4$. Center-Left: a $5:9$
resonant orbit for $n=6$. Center-Right: a $1:3$ resonant orbit for $n=6$.
Bottom-Left: a $3:4$ resonant orbit for $n=8$. Bottom-Right: a $1:4$ resonant
orbit for $n=8$. }%
\label{fig09}%
\end{figure}
\clearpage

\section{Other configurations}

\subsection{The Maxwell configuration}

The choreographies in the preceding sections are unstable, in part because
they arise directly or indirectly from an unstable relative equilibrium. To
determine stable solutions it is helpful to consider orbits that emanate from
a stable relative equilibrium. The polygonal equilibrium is never stable;
about half of its eigenvalues are stable and about half are unstable. For this
reason we now consider the Maxwell configuration, consisting of an $n$-polygon
with an additional massive body at the center, which is known to be stable
when $n\geq7$. Specifically, the central body has mass $m_{0}=\mu$, and the
other $n$ bodies have equal mass $m_{j}=1$ for $j\in\{1,...,n\}$. Let
$(u_{j},z_{j})\in\mathbb{C}\times\mathbb{R}$ be the position of body
$j\in\{0,1,...,n\}$. The Newton equations of motion for the $n+1$ bodies in
rotating coordinates
\[
q_{j}(t)=(e^{i\sqrt{\omega}t}u_{j}(t),z_{j}(t))
\]
have an equilibrium with $\left(  u_{0},z_{0}\right)  =\left(  0,0\right)  $
and $\left(  u_{j},z_{j}\right)  =\left(  e^{ij\zeta},0\right)  $ for
$j\in\{1,...,n\}$, where $\zeta=2\pi/n$ and%
\[
\omega=\mu+s_{1}\text{.}%
\]
This well-known Maxwell configuration reduces to the polygonal relative
equilibrium when $\mu=0$.

For $n\geq7$ all planar eigenvalues are imaginary, and produce Planar Lyapunov
families. The $n+1$ spatial eigenvalues include $0$ (due to symmetries),
$i\sqrt{\mu+n}$ for $k=n$, and%
\[
i\sqrt{\mu+s_{k}},\qquad k=1,~\cdots~,n-1\text{.}%
\]
The frequency $\sqrt{\mu+n}$ produces the Vertical Lyapunov family, which
corresponds to the oscillatory ring in \cite{MeSc93}. For $k=n/2$, with $n$
even, we obtain a Hip-Hop family \cite{MeSc93}. For the Maxwell configuration
we say that a Lyapunov orbit is $\ell:m$ resonant when its period satisfies
\[
T_{\ell:m}=\frac{2\pi}{\sqrt{\mu+s_{1}}}\frac{\ell}{m},
\]
where $\ell$ and $m$ are relatively prime such that $k\ell-m\in n\mathbb{Z}$.
For an $\ell:m$ resonant Lyapunov orbit the $n$ bodies of equal mass follow
the same path as in Theorem \ref{proposition}.

To illustrate our numerical computations we chose the first stable case for
$\mu=200$, namely $n=7$. We also consider the case $n=8$ with $\mu=300$. We
computed many families for $n=7$ and $n=8$ and we present only a few planar
resonant orbits in Figure~\ref{fig11} and spatial resonant orbits in
Figure~\ref{fig12}.

\subsection{A triangular configuration}

Here we present some families of periodic solutions that emanate from the
\textquotedblleft triangular\textquotedblright\ equilibrium shown in the
top-left panel of Figure~\ref{fig12}, with $9$ bodies of equal mass. Periodic
solutions that emanate from the triangular equilibrium have been determined
with the same numerical scheme used throughout this paper. However, a detailed
description of these results is outside the scope of the current paper, whose
aim is the continuation of solutions with symmetries that produce choreographies.

The triangular equilibrium can been reached by following one of the families
of spatial periodic orbits that bifurcate from the polygonal relative
equilibrium for $n=9$. These spatial solutions have the symmetry%
\[
u_{j}(t)=\bar{u}_{n-j}(-t)\text{.}%
\]
Actually, in addition to the Vertical Lyapunov families that produce
choreographies from the polygonal configuration, we have also determined these
solutions, which do not produce choreographies. To the best of our knowledge,
the existence of these families has not been established before.

\vskip0.25cm \textbf{Acknowledgements.} We would like to thank R. Montgomery,
J. Montaldi, D. Ayala and L. Garc\'{\i}a-Naranjo for many interesting
discussions. We also acknowledge the assistance of Ramiro Chavez Tovar with
the preparation of figures and animations.
\begin{figure}[h]
\par
\begin{center}
\resizebox{15.1cm}{!}{
\includegraphics{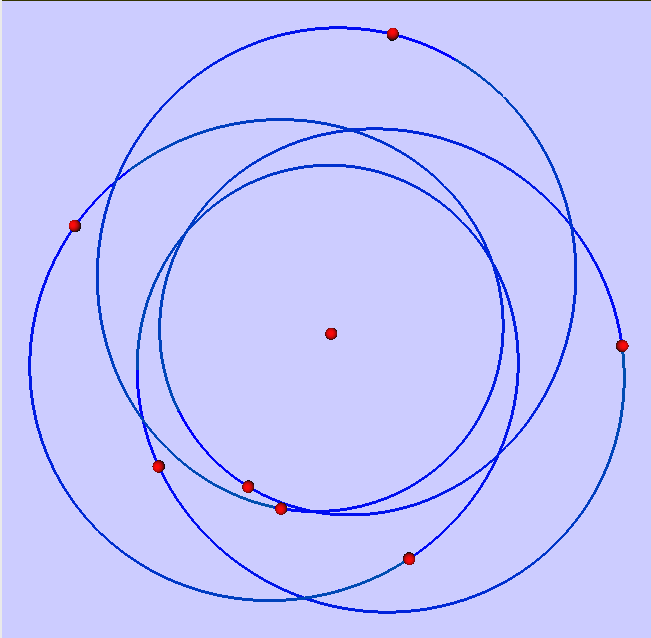} ~~~
\includegraphics{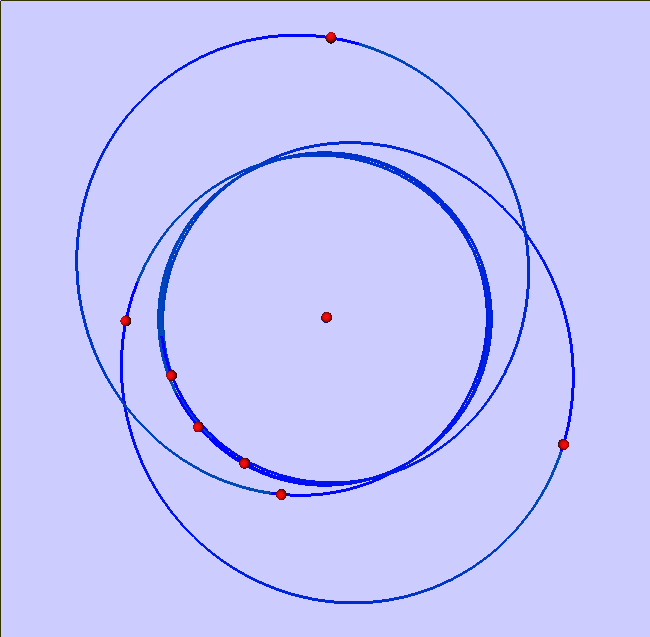}} \vskip0.15cm
\resizebox{15.1cm}{!}{
\includegraphics{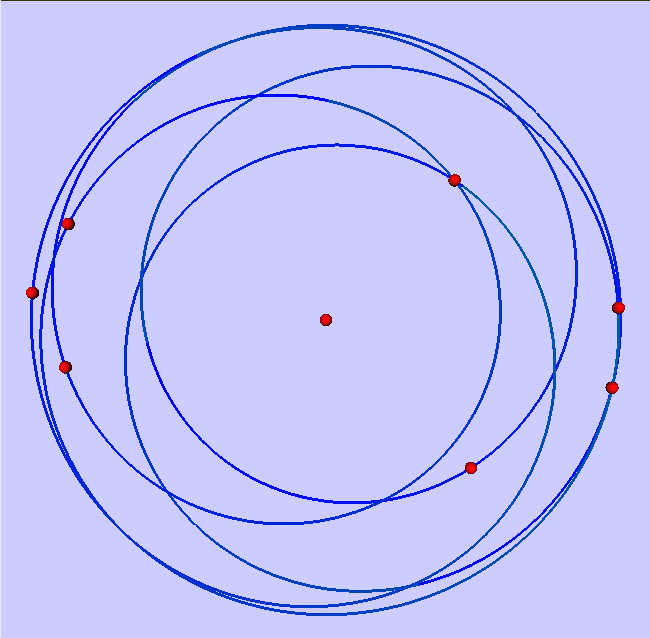} ~~~
\includegraphics{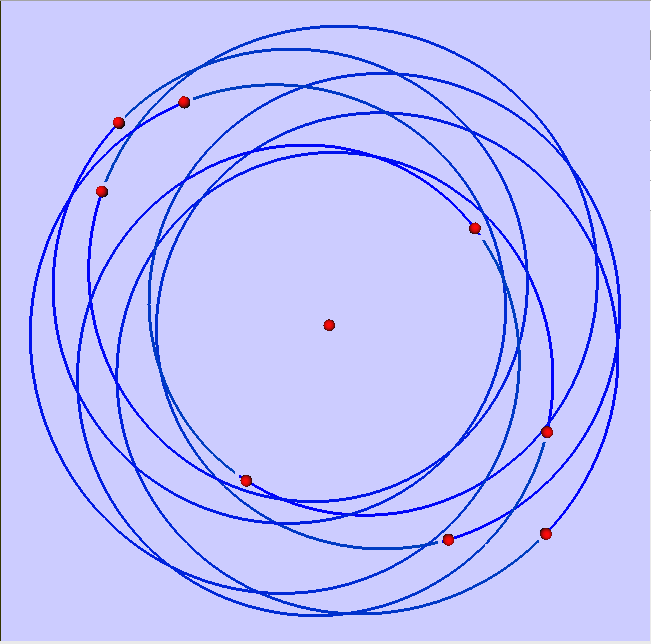}} \vskip0.15cm
\resizebox{15.1cm}{!}{
\includegraphics{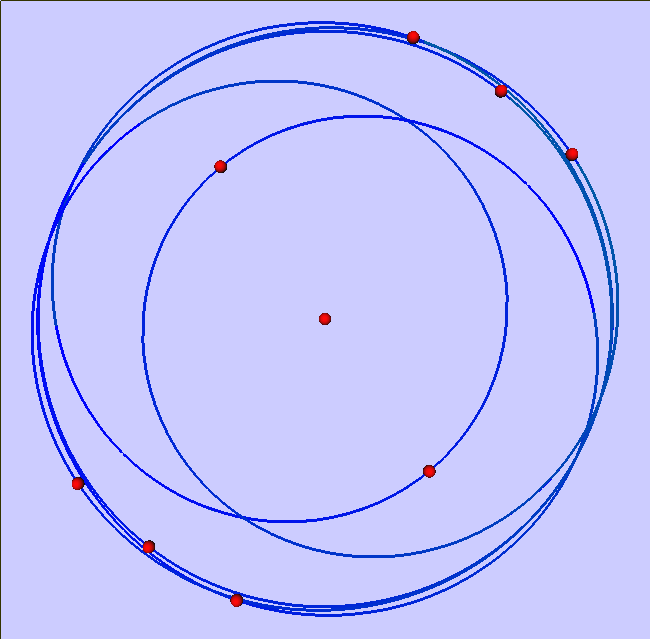} ~~~
\includegraphics{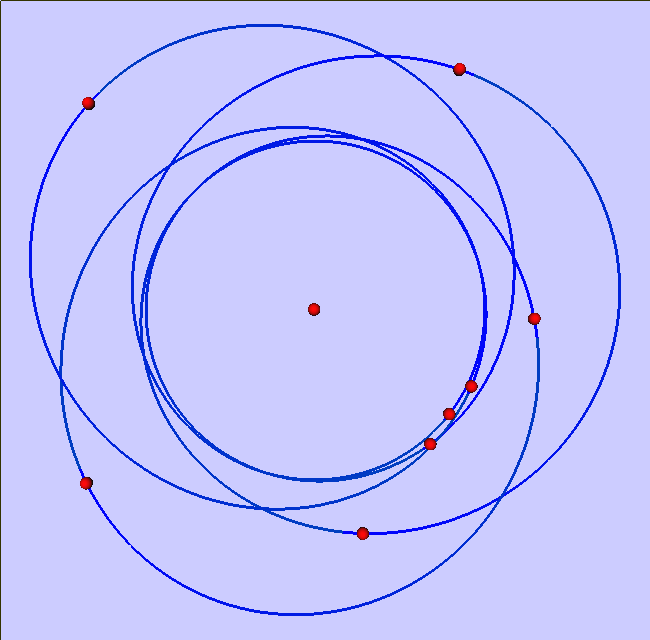}}
\end{center}
\caption{ Top-Left: a $4:3$ resonant orbit for $7+1$ bodies and $k=6$.
Top-Right: a $5:2$ resonant orbit for $7+1$ bodies and $k=6$. Center-Left: a
$5:3$ resonant orbit for $7+1$ bodies and $k=2$. Center-Right: a $7:3$
resonant orbit for $8+1$ bodies for $k=8$. Bottom-Left: a $5:2$ resonant orbit
for $8+1$ bodies and $k=2$. Bottom-Right: a $5:3$ resonant orbit for $8+1$
bodies and $k=7$. }%
\label{fig10}%
\end{figure}
\clearpage
\begin{figure}[h]
\par
\begin{center}
\resizebox{15.1cm}{!}{
\includegraphics{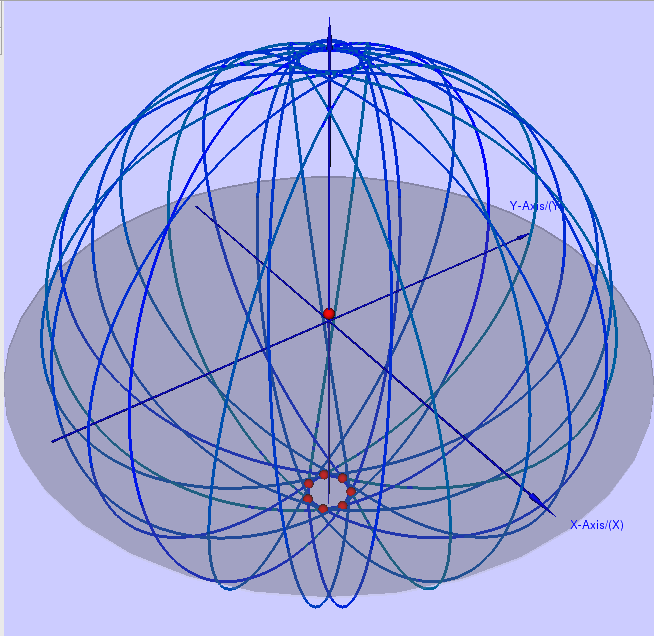} ~~~
\includegraphics{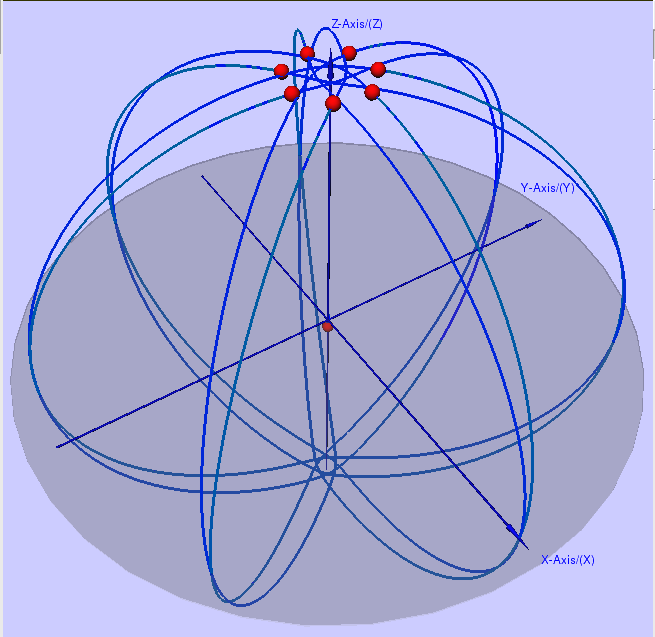}} \vskip0.15cm
\resizebox{15.1cm}{!}{
\includegraphics{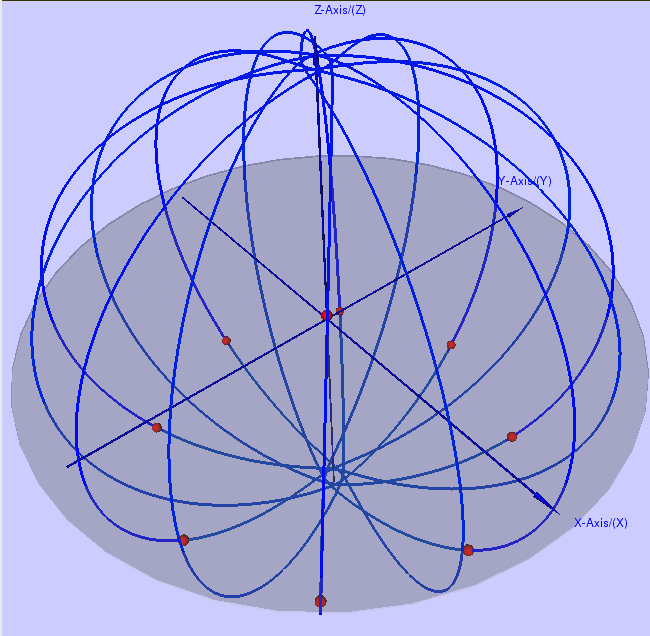} ~~~
\includegraphics{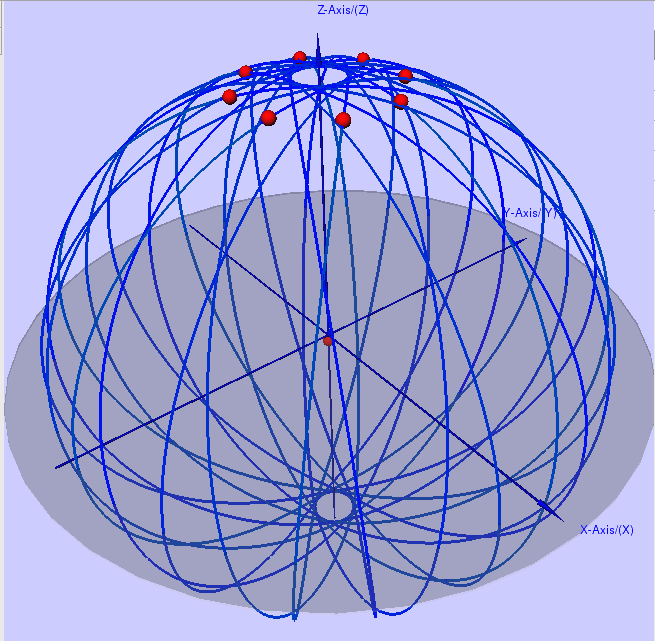}} \vskip0.15cm
\resizebox{15.1cm}{!}{
\includegraphics{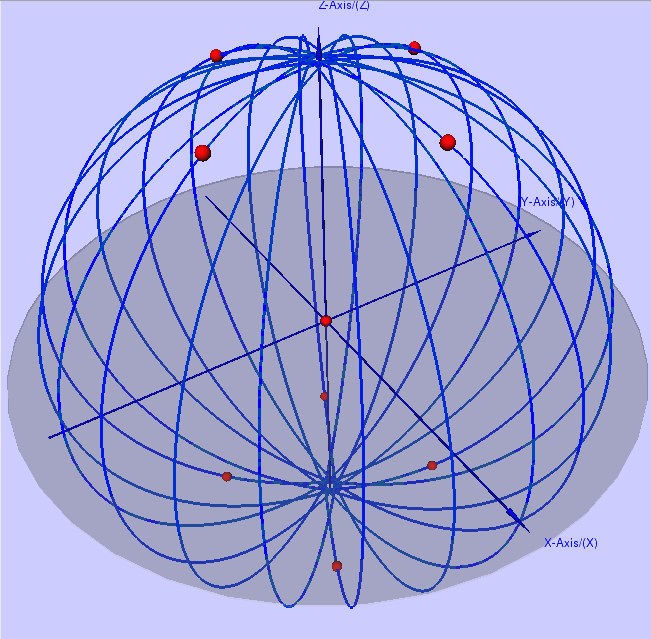} ~~~
\includegraphics{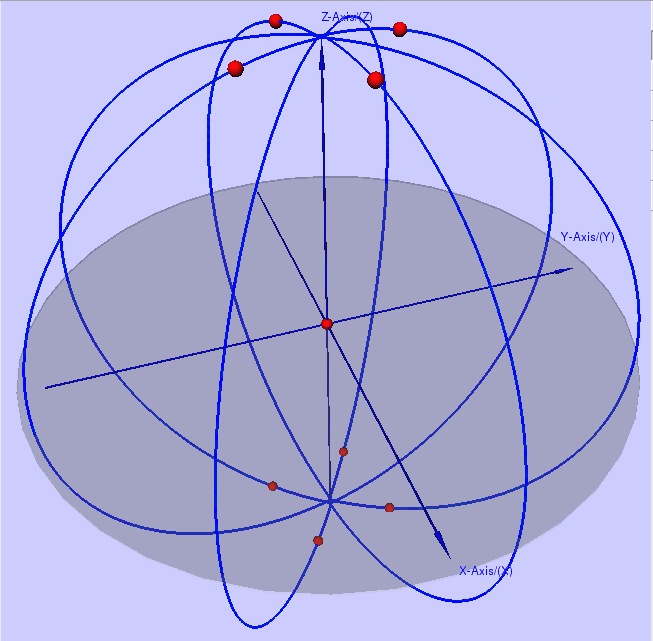}}
\end{center}
\caption{ Top-Left: a $15:14$ resonant orbit for $7+1$ bodies and $k=7$.
Top-Right: an $8:7$ resonant orbit for $7+1$ bodies and $k=7$. Center-Left: a
$9:8$ resonant orbit for $8+1$ bodies and $k=8$. Center-Right:a $17:16$
resonant orbit for $8+1$ bodies and $k=8$. Bottom-Left: a $13:12$ resonant
orbit for $8+1$ bodies and $k=4$. Bottom-Right: a $15:12$ resonant orbit for
$8+1$ bodies and $k=4$. }%
\label{fig11}%
\end{figure}
\clearpage
\begin{figure}[h]
\par
\begin{center}
\resizebox{15.2cm}{!}{
\includegraphics{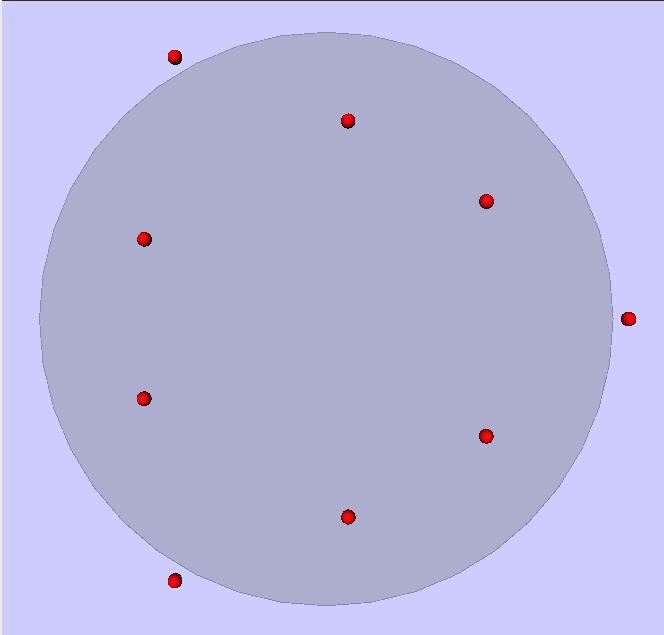} ~~~
\includegraphics{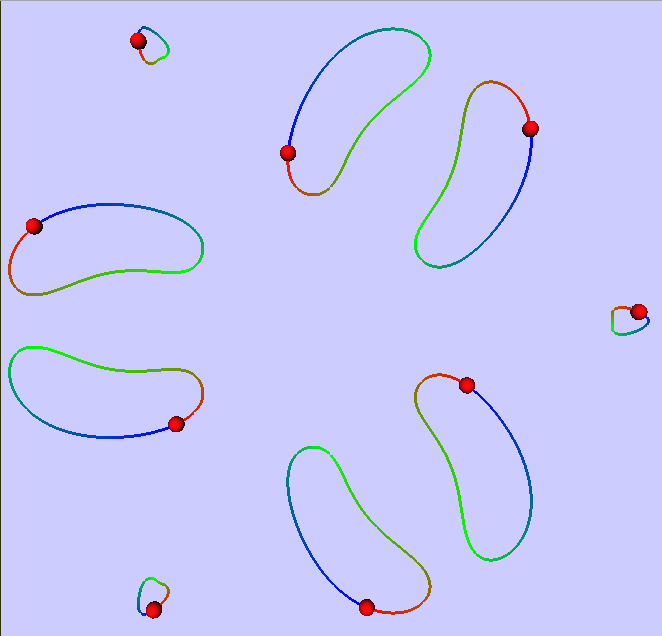}} \vskip0.15cm \resizebox{15.2cm}{!}{
\includegraphics{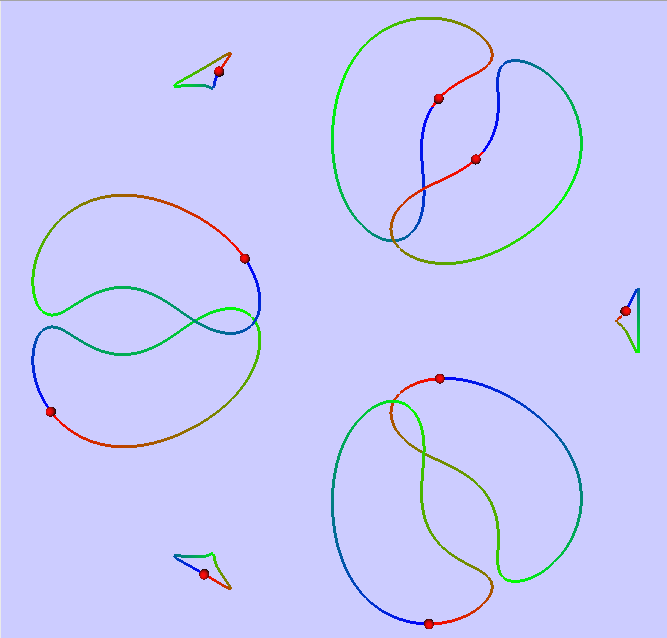} ~~~
\includegraphics{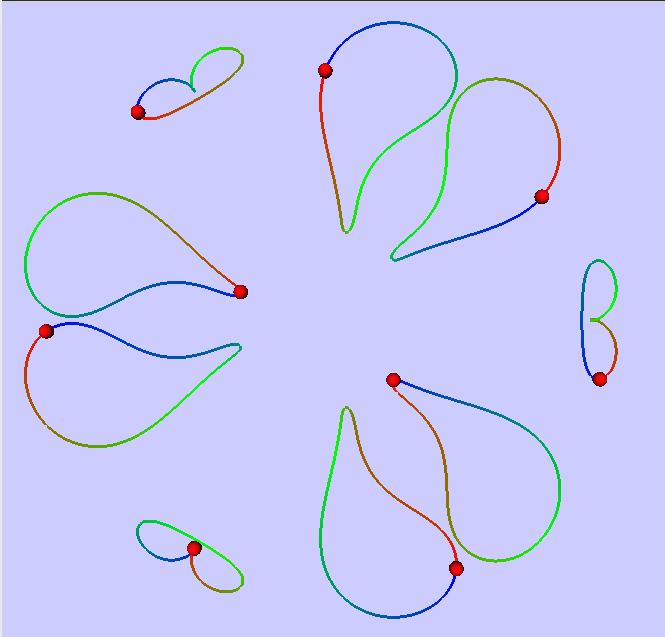}} \vskip0.15cm \resizebox{15.2cm}{!}{
\includegraphics{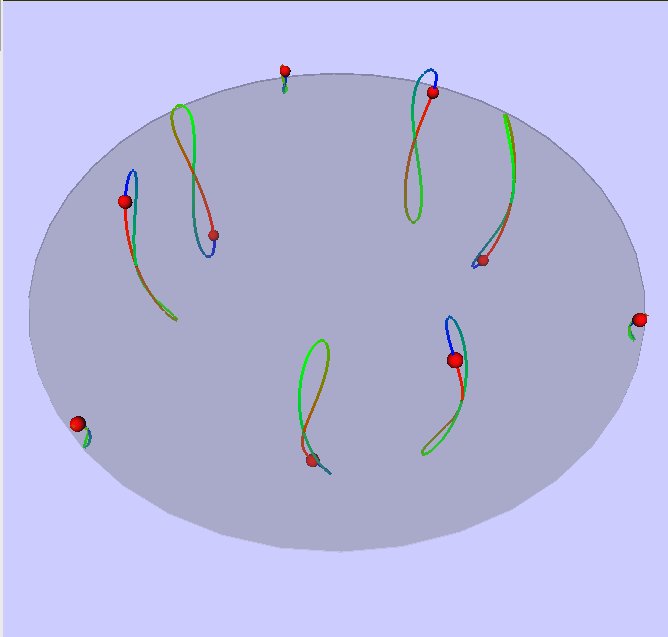} ~~~
\includegraphics{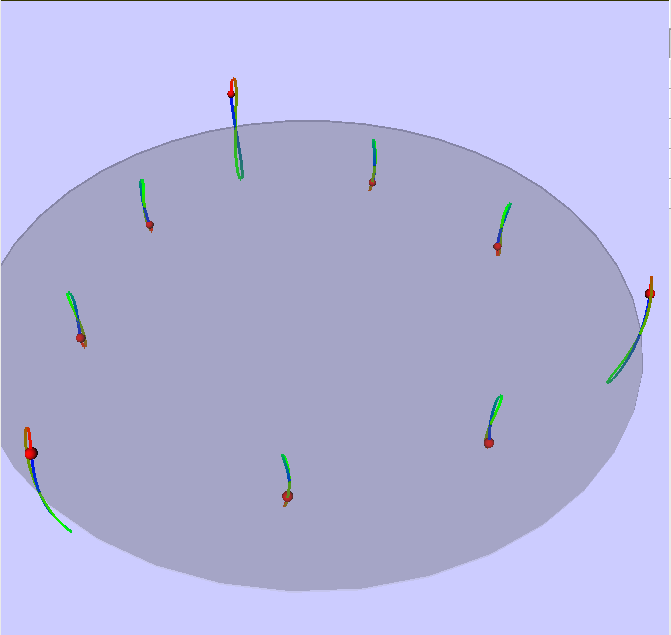}}
\end{center}
\caption{ Some Lyapunov families with different symmetries for $n=9$.
Top-Left: another equilibrium of the 9-body problem. Top-Right: an orbit along
a bifurcating Planar family. Center-Left: an orbit along another bifurcating
Planar family. Center-Right: an orbit along yet another bifurcating Planar
family. Bottom-Left: an orbit along a bifurcating spatial family.
Bottom-Right: an orbit along another bifurcating spatial family. }%
\label{fig12}%
\end{figure}


\end{document}